\documentclass[10pt]{amsart}
\usepackage{amsthm}
\usepackage{amsmath,amssymb,amsfonts,marvosym,amscd}
\usepackage{wasysym}  
\usepackage{pifont} 
\usepackage{graphicx}
\usepackage{psfrag}
\usepackage{datetime}

\usepackage{tikz,tikz-cd}  
\usetikzlibrary{decorations.pathreplacing}

\usepackage{geometry} 
\usepackage[all]{xy}

\usepackage{hyperref}  

\let\emptyset\varnothing  

\title{Higher Frieze Patterns}
\author{Jordan McMahon}

\begin{document}

\newtheorem{lm}{Lemma}[section]
\newtheorem{prop}[lm]{Proposition}
\newtheorem{conj}[lm]{Conjecture}
\newtheorem{eg}[lm]{Example}
\newtheorem*{claim}{Claim}
\newtheorem{satz}[lm]{Satz}

\newtheorem*{corollary}{Corollary}
\newtheorem{cor}[lm]{Corollary}
\newtheorem{theorem}[lm]{Theorem}
\newtheorem*{thm}{Theorem}

\theoremstyle{definition}
\newtheorem{defn}{Definition}[section]
\newtheorem*{defini}{Definition}
\newtheorem*{definitionen}{Definitionen}
\newtheorem{ueb}{\"Ubungsbeispiel}
\newtheorem{bsp}[lm]{Beispiel}
\newtheorem{bspe}[lm]{Beispiele}
\newtheorem*{exas}{Beispiele}
\newtheorem*{eigen}{Eigenschaften}
\newtheorem*{rem}{Remark}
\newtheorem{remark}{Remark}
\newtheorem{qu}{Question}
\newcommand{\bea}{\begin{eqnarray*}}
\newcommand{\eea}{\end{eqnarray*}}

\begin{abstract}
Frieze patterns have an interesting combinatorial structure, which has proven very useful in the study of cluster algebras. We introduce $(k,n)$-frieze patterns, a natural generalisation of the classical notion. A generalisation of the bijective correspondence between frieze patterns of width $n$ and clusters of Pl\"ucker coordinates in the cluster structure of the Grassmannian $\mathrm{Gr}(2,n+3)$ is obtained.
\end{abstract}

\maketitle
\tableofcontents

\newcommand{\jm}[1]{\textcolor{brown}{#1}}   
\newcommand{\kb}[1]{\textcolor{cyan}{#1}}    
\newcommand{\phd}[1]{\textcolor{orange}{#1}}  
\definecolor{light-gray}{gray}{0.8}
\newcommand{\rd}[1]{\textcolor{red}{#1}}  

\section{Introduction}

Frieze patterns were introduced by Coxeter in 1971 \cite{cox}, and a bijection between frieze patterns of width $n$ and triangulations of $(n+3)$-gons was established by Conway and Coxeter shortly thereafter \cite{cc1}\cite{cc2}. This impressive result received wider interest with the advent of cluster algebras, introduced by Fomin and Zelevinsky in \cite{fz1}\cite{fz2}. Frieze patterns satisfy many nice combinatorial properties, such as being invariant under a glide reflection, and are therefore  periodic. There already exist many generalisations of frieze patterns, see \cite{mg} for a comprehensive introduction to the subject. 

A first motivating example of a cluster algebra was given by the homogeneous coordinate ring $\mathbb{C}[\mathrm{Gr}(2,n+3)]$ of the Grassmannian of 2-dimensional subspaces in $\mathbb{C}^{n+3}$, the ring being generated by Pl\"ucker coordinates, and subject to Pl\"ucker relations. Each cluster of Pl\"ucker coordinates in the cluster structure of $\mathbb{C}[\mathrm{Gr}(2,n+3)]$ corresponds to a cluster-tilted algebra of Dynkin type $A_n$ \cite[Prop 12.7]{fz2}. In particular, each cluster in this cluster algebra corresponds to a triangulation of an $(n+3)$-gon, just as each frieze pattern of width $n$ does. This connection was formalised in \cite{cch}, where the Caldero-Chapoton formula was introduced for clusters and found to determine the entries in the corresponding frieze pattern. 

However, while cluster algebras for Grassmannians $\mathrm{Gr}(k,n)$ are well understood in the case that $k=2$, they are by no means limited to this case. In \cite{sco}, Scott proved that the homogeneous coordinate ring $\mathbb{C}[\mathrm{Gr}(k,n)]$ is a cluster algebra for all values of $2\leq k\leq n/2$. We denote the set of clusters of Pl\"ucker coordinates in the cluster structure of $\mathbb{C}[\mathrm{Gr}(k,n)]$ by $\mathcal{A}_{k,n}$ and throughout this paper we will consider $k$ and $n$ to be integers with $2\leq k\leq n/2$,

The motivation for this paper is to fully describe the clusters in $\mathcal{A}_{k,n}$ with $k\geq 3$ in terms of frieze patterns. To this purpose, we introduce \emph{$(k,n)$-frieze patterns} or \emph{higher frieze patterns}. However it turns out that there are significantly more $(k,n)$-frieze patterns than clusters in $\mathcal{A}_{k,n}$. Therefore, we primarily consider \emph{geometric $(3,n)$-frieze patterns}, which we will show are precisely those higher frieze patterns that correspond to a cluster in $\mathcal{A}_{3,n}$, and moreover satisfy a generalised version of the unimodular rule that defines the classical frieze patterns. 

In the classical case, a frieze of width $n$ corresponds to a cluster-tilting object of type $A_n$; this leads to a more general definition of frieze patterns as functions on a repetition quiver, see \cite{ars} or \cite{mg} for details. A similar approach may be considered for higher frieze patterns using the higher Auslander-Reiten theory introduced by Iyama in \cite{iy1}, \cite{iy2}. In particular, a higher frieze pattern may be seen as a function on the cylinder of a higher Auslander algebra of type $A$, invariant under a higher-dimensional glide reflection. To find where higher frieze patterns lie in the pantheon of generalisations of frieze patterns, we show that each $(k,n)$-frieze pattern determines an $\mathrm{SL}_k$-frieze pattern of width $n-k-1$. This complements a result of \cite[Proposition 3.2.1]{mgost}, that each $\mathrm{SL}_3$-frieze pattern of width $n-k-1$ is found to determine a point in the Grassmannian $\mathrm{Gr}(k,n)$. 

We remark here that Oppermann and Thomas \cite{ot} found a generalisation of the bijection between triangulations of $(n+3)$-gons and cluster-tilting objects of type $A_n$ for higher cluster-tilting theory using triangulations of cyclic polytopes. In a sequel paper \cite{mc2}, we will instead associate a particular class of clusters in $\mathcal{A}_{k,n}$ with superimposed triangulations. This illustrates contrasting combinatorics for these two generalisations of the combinatorial model for cluster algebras of type $A$.

\section{Coxeter's Frieze Patterns}

In the sense of Coxeter, a \emph{frieze pattern} \cite{cox} is an array of numbers satisfying the following conditions:
\begin{itemize}
\item The array has finitely many rows (though infinitely many columns are needed)
\item The two uppermost and bottommost rows are fixed such the first and final rows consist of only 0's; the second and penultimate rows are rows consisting of only 1's.
\item Consecutive rows are displayed with a shift, and every diamond $$\begin{tikzpicture}
\node(a) at (0,0){$a$};
\node(b) at (-0.5,-0.5){$b$};
\node(c) at (0.5,-0.5){$c$};
\node(d) at (0,-1){$d$};
\end{tikzpicture}$$
satisfies the \emph{unimodular rule}: $bc-ad=1$. 
\end{itemize}
Note that we omit, by convention, the top and bottom rows of zeroes from the frieze pattern. Such an array satisfying instead that every $k\times k$-minor has determinant one, in place of the unimodular rule, is called an \emph{$\mathrm{SL}_k$-frieze pattern} \cite{cr} (see also \cite{mg}). An $\mathrm{SL}_2$-frieze pattern is another name for one of Coxeter's frieze patterns. 
\begin{eg}
Examples of frieze patterns (in the sense of Coxeter) include:
$$\begin{tikzpicture}
\node(a) at (4,0){$1$};
\node(b) at (0.4,0){$1$};
\node(c) at (0.8,0){$1$};
\node(d) at (1.2,0){$1$};
\node(e) at (1.6,0){$1$};
\node(f) at (2,0){$1$};
\node(g) at (2.4,0){$1$};
\node(h) at (2.8,0){$1$};
\node(i) at (3.2,0){$1$};
\node(j) at (3.6,0){$1$};
\node(oa) at (0.2,0.4){$2$};
\node(ob) at (0.6,0.4){$1$};
\node(oc) at (1,0.4){$2$};
\node(od) at (1.4,0.4){$1$};
\node(oe) at (1.8,0.4){$2$};
\node(of) at (2.2,0.4){$1$};
\node(og) at (2.6,0.4){$2$};
\node(oh) at (3,0.4){$1$};
\node(oi) at (3.4,0.4){$2$};
\node(oj) at (3.8,0.4){$1$};
\node(pa) at (0,0.8){$1$};
\node(pb) at (0.4,0.8){$1$};
\node(pc) at (0.8,0.8){$1$};
\node(pd) at (1.2,0.8){$1$};
\node(pe) at (1.6,0.8){$1$};
\node(pf) at (2,0.8){$1$};
\node(pg) at (2.4,0.8){$1$};
\node(ph) at (2.8,0.8){$1$};
\node(pi) at (3.2,0.8){$1$};
\node(pj) at (3.6,0.8){$1$};
\end{tikzpicture}$$

$$\begin{tikzpicture}
\node(a) at (4,0){$1$};
\node(b) at (4.4,0){$1$};
\node(c) at (0.8,0){$1$};
\node(d) at (1.2,0){$1$};
\node(e) at (1.6,0){$1$};
\node(f) at (2,0){$1$};
\node(g) at (2.4,0){$1$};
\node(h) at (2.8,0){$1$};
\node(i) at (3.2,0){$1$};
\node(j) at (3.6,0){$1$};
\node(oa) at (4.2,0.4){$1$};
\node(ob) at (0.6,0.4){$3$};
\node(oc) at (1,0.4){$1$};
\node(od) at (1.4,0.4){$2$};
\node(oe) at (1.8,0.4){$2$};
\node(of) at (2.2,0.4){$1$};
\node(og) at (2.6,0.4){$3$};
\node(oh) at (3,0.4){$1$};
\node(oi) at (3.4,0.4){$2$};
\node(oj) at (3.8,0.4){$2$};
\node(pa) at (4,0.8){$1$};
\node(pb) at (0.4,0.8){$2$};
\node(pc) at (0.8,0.8){$2$};
\node(pd) at (1.2,0.8){$1$};
\node(pe) at (1.6,0.8){$3$};
\node(pf) at (2,0.8){$1$};
\node(pg) at (2.4,0.8){$2$};
\node(ph) at (2.8,0.8){$2$};
\node(pi) at (3.2,0.8){$1$};
\node(pj) at (3.6,0.8){$3$};
\node(qa) at (0.2,1.2){$1$};
\node(qb) at (0.6,1.2){$1$};
\node(qc) at (1,1.2){$1$};
\node(qd) at (1.4,1.2){$1$};
\node(qe) at (1.8,1.2){$1$};
\node(qf) at (2.2,1.2){$1$};
\node(qg) at (2.6,1.2){$1$};
\node(qh) at (3,1.2){$1$};
\node(qi) at (3.4,1.2){$1$};
\node(qj) at (3.8,1.2){$1$};
\end{tikzpicture}$$

$$\begin{tikzpicture}
\node(a) at (4,0){$1$};
\node(b) at (4.4,0){$1$};
\node(c) at (0.8,0){$1$};
\node(d) at (1.2,0){$1$};
\node(e) at (1.6,0){$1$};
\node(f) at (2,0){$1$};
\node(g) at (2.4,0){$1$};
\node(h) at (2.8,0){$1$};
\node(i) at (3.2,0){$1$};
\node(j) at (3.6,0){$1$};
\node(oa) at (4.2,0.4){$1$};
\node(ob) at (0.6,0.4){$3$};
\node(oc) at (1,0.4){$1$};
\node(od) at (1.4,0.4){$3$};
\node(oe) at (1.8,0.4){$1$};
\node(of) at (2.2,0.4){$3$};
\node(og) at (2.6,0.4){$1$};
\node(oh) at (3,0.4){$3$};
\node(oi) at (3.4,0.4){$1$};
\node(oj) at (3.8,0.4){$3$};
\node(pa) at (4,0.8){$2$};
\node(pb) at (0.4,0.8){$2$};
\node(pc) at (0.8,0.8){$2$};
\node(pd) at (1.2,0.8){$2$};
\node(pe) at (1.6,0.8){$2$};
\node(pf) at (2,0.8){$2$};
\node(pg) at (2.4,0.8){$2$};
\node(ph) at (2.8,0.8){$2$};
\node(pi) at (3.2,0.8){$2$};
\node(pj) at (3.6,0.8){$2$};
\node(qa) at (0.2,1.2){$3$};
\node(qb) at (0.6,1.2){$1$};
\node(qc) at (1,1.2){$3$};
\node(qd) at (1.4,1.2){$1$};
\node(qe) at (1.8,1.2){$3$};
\node(qf) at (2.2,1.2){$1$};
\node(qg) at (2.6,1.2){$3$};
\node(qh) at (3,1.2){$1$};
\node(qi) at (3.4,1.2){$3$};
\node(qj) at (3.8,1.2){$1$};
\node(ra) at (0,1.6){$1$};
\node(rb) at (0.4,1.6){$1$};
\node(rc) at (0.8,1.6){$1$};
\node(rd) at (1.2,1.6){$1$};
\node(re) at (1.6,1.6){$1$};
\node(rf) at (2,1.6){$1$};
\node(rg) at (2.4,1.6){$1$};
\node(rh) at (2.8,1.6){$1$};
\node(ri) at (3.2,1.6){$1$};
\node(rj) at (3.6,1.6){$1$};
\end{tikzpicture}$$

$$\begin{tikzpicture}
\node(a) at (4,0){$1$};
\node(b) at (4.4,0){$1$};
\node(c) at (0.8,0){$1$};
\node(d) at (1.2,0){$1$};
\node(e) at (1.6,0){$1$};
\node(f) at (2,0){$1$};
\node(g) at (2.4,0){$1$};
\node(h) at (2.8,0){$1$};
\node(i) at (3.2,0){$1$};
\node(j) at (3.6,0){$1$};
\node(oa) at (4.2,0.4){$2$};
\node(ob) at (0.6,0.4){$4$};
\node(oc) at (1,0.4){$1$};
\node(od) at (1.4,0.4){$2$};
\node(oe) at (1.8,0.4){$2$};
\node(of) at (2.2,0.4){$2$};
\node(og) at (2.6,0.4){$1$};
\node(oh) at (3,0.4){$4$};
\node(oi) at (3.4,0.4){$1$};
\node(oj) at (3.8,0.4){$2$};
\node(pa) at (4,0.8){$3$};
\node(pb) at (0.4,0.8){$3$};
\node(pc) at (0.8,0.8){$3$};
\node(pd) at (1.2,0.8){$1$};
\node(pe) at (1.6,0.8){$3$};
\node(pf) at (2,0.8){$3$};
\node(pg) at (2.4,0.8){$1$};
\node(ph) at (2.8,0.8){$3$};
\node(pi) at (3.2,0.8){$3$};
\node(pj) at (3.6,0.8){$1$};
\node(qa) at (0.2,1.2){$2$};
\node(qb) at (0.6,1.2){$2$};
\node(qc) at (1,1.2){$2$};
\node(qd) at (1.4,1.2){$1$};
\node(qe) at (1.8,1.2){$4$};
\node(qf) at (2.2,1.2){$1$};
\node(qg) at (2.6,1.2){$2$};
\node(qh) at (3,1.2){$2$};
\node(qi) at (3.4,1.2){$2$};
\node(qj) at (3.8,1.2){$1$};
\node(ra) at (0,1.6){$1$};
\node(rb) at (0.4,1.6){$1$};
\node(rc) at (0.8,1.6){$1$};
\node(rd) at (1.2,1.6){$1$};
\node(re) at (1.6,1.6){$1$};
\node(rf) at (2,1.6){$1$};
\node(rg) at (2.4,1.6){$1$};
\node(rh) at (2.8,1.6){$1$};
\node(ri) at (3.2,1.6){$1$};
\node(rj) at (3.6,1.6){$1$};
\end{tikzpicture}$$
\end{eg}

A frieze pattern is said to be of width $n$ if it has $n$ rows strictly between the border rows of ones at the top and bottom. 

\section{Background}

\subsection{Pl\"ucker Relations}
Recall that the Grassmannian of all $k$-dimensional subspaces of $\mathbb{C}^n$, $\mathrm{Gr}(k,n)$, can be embedded into the projective space $\mathbb{P}(\wedge^k(\mathbb{C}^n))$ via the Pl\"ucker embedding. The coordinates of $\wedge^k(\mathbb{C}^n)$ are called the \emph{Pl\"ucker coordinates} and are indexed by the $k$-multisets  $I=\{i_1,i_2,\cdots,i_k\}$ with elements from $\{1,\cdots ,n\}.$ The coordinate defined by $\{i_1,i_2,\cdots,i_k\}$ will be denoted $p_{i_1i_2\cdots i_k}$. 

In general, the ordering of elements in a Pl\"ucker coordinate may not be known;
the definition of the Pl\"ucker coordinates may be extended to allow for this. By convention, the Pl\"ucker coordinates possess an antisymmetry:
\begin{align} p_{i_1\cdots i_ri_s\cdots i_k}=-p_{i_1\cdots i_si_r\cdots i_k}\label{eq1}.\end{align}
 In particular, if $i_r=i_s$ then $p_{i_1\cdots i_ri_r\cdots i_k}=0.$
The Pl\"ucker embedding satisfies the determinantal \emph{Pl\"ucker relations}. In the case $k=2$, the Pl\"ucker relations are

$$p_{ac}p_{bd}=p_{ab}p_{cd}+p_{ad}p_{bc},$$
where $1\leq a<b<c<d\leq n$. Now, let $p_{Ii_1i_2\cdots i_r}$ be the Pl\"ucker coordinate for the subset $I\cup\{i_1\}\cup\{i_2\}\cup\cdots\cup\{i_r\}$. 
Then, more generally, the Pl\"ucker relations for $k>2$ are $$p_{Iac}p_{Ibd}=p_{Iab}p_{Icd}+p_{Iad}p_{Ibc},$$ where $I$ is a $(k-2)$-subset of $\{1,\cdots, n\}$  with $\{a,b,c,d\}\cap I=\emptyset$. Another set of generating relations for the $(k,n)$-Pl\"ucker relations is the set of relations

\begin{align}\label{eq3}
\sum_{r=0}^k(-1)^np_{i_1i_2\cdots i_{k-1}j_r}p_{j_0\cdots \widehat{j_r}\cdots j_k}=0,
\end{align}
 where $1\leq i_0<i_1<\cdots<i_{k-1}\leq n$ and $1\leq j_0< j_1< \cdots j_k\leq n$. See for example \cite{mar} for a reference on Pl\"ucker relations.

\subsection{Higher-Dimensional Auslander-Reiten Theory}

In the context of generalising classical Auslander-Reiten theory to higher dimensions, Iyama introduced in \cite{iy2} the notion of a higher Auslander algebra. For a quiver $Q$ of Dynkin type, there is an explicit description of the quiver of the $n$-Auslander algebra of the path algebra $\mathbb{C}Q$. Let $Q$ be the following quiver:
$$\begin{tikzpicture}
\node(a) at (-2,0){1};
\node(b) at (-1,0){2};
\node(c) at (0,0){$\cdots$};
\node(d) at (1,0){$d$};

\draw[->](a) edge(b);
\draw[->](b) edge(c);
\draw[->](c) edge(d);
\end{tikzpicture}$$
Set $e_i$ to be the $\mathbb{Z}^n$ vector with 1 in the $i$-th coordinate, and 0 in every other coordinate. Then let $$v_i=\begin{cases} -e_i & i=1,\\ e_{i-1} -e_{i} & 2\leq i\leq n.\end{cases}$$
The quiver of the $m$-Auslander algebra $A_d^{(m)}$ can be described as follows.

\begin{itemize}
\item The vertices are $$(Q^{(m)}_d)_0=\{(l_1,l_2,\cdots,l_{m+1})\in \mathbb{Z}^{m+1}|l_1\geq 1;l_2,\cdots,l_{m+1}\geq 0;l_1+l_2+\cdots l_{m+1}\leq d\}.$$
\item For each vertex $\underline{l}\in Q^{(m)}_0$, there is an arrow $\underline{l}\rightarrow \underline{l}+v_i$ wherever $\underline{l}+v_i$ is in $(Q^{(m)}_d)_0$.
\end{itemize}
For example, the quiver of the $1$-Auslander algebra $A^{(1)}_4$ is:
$$\begin{tikzpicture}[xscale=5,yscale=2.5]
\node(xa) at (-2,1){$13$};
\node(xb) at (-1.6,1){$12$};
\node(xc) at (-1.2,1){$11$};
\node(xd) at (-0.8,1){$10$};

\node(xab) at (-1.8,1.2){$22$};
\node(xac) at (-1.4,1.2){$21$};
\node(xad) at (-1,1.2){$20$};

\node(xbb) at (-1.6,1.4){$31$};
\node(xbc) at (-1.2,1.4){$30$};

\node(xcc) at (-1.4,1.6){$40$};

\draw[->](xa) edge(xab);
\draw[->](xb) edge(xac);
\draw[->](xc) edge(xad);

\draw[->](xbb) edge(xac);
\draw[->](xbc) edge(xad);

\draw[->](xbb) edge(xcc);

\draw[->](xab) edge(xb);
\draw[->](xac) edge(xc);
\draw[->](xad) edge(xd);

\draw[->](xab) edge(xbb);
\draw[->](xac) edge(xbc);

\draw[->](xcc) edge(xbc);
\end{tikzpicture}$$
The quiver of the $2$-Auslander algebra $A^{(2)}_4$ is:

$$\begin{tikzpicture}[xscale=6,yscale=2]

\node(a) at (-2,0){$130$};
\node(b) at (-1.6,0){$120$};
\node(c) at (-1.2,0){$110$};
\node(d) at (-0.8,0){$100$};
\node(ab) at (-1.8,0.2){$220$};
\node(ac) at (-1.4,0.2){$210$};
\node(ad) at (-1,0.2){$200$};
\node(bb) at (-1.6,0.4){$310$};
\node(bc) at (-1.2,0.4){$300$};
\node(cc) at (-1.4,0.6){$400$};
\draw[->](a) edge (ab);
\draw[->](b) edge (ac);
\draw[->](c) edge (ad);
\draw[->](bb) edge (ac);
\draw[->](bc) edge (ad);
\draw[->](bb) edge (cc);
\draw[->](ab) edge (b);
\draw[->](ac) edge (c);
\draw[->](ad) edge (d);
\draw[->](ab) edge (bb);
\draw[->](ac) edge (bc);
\draw[->](cc) edge (bc);
\node(xa) at (-1.8,-1){$121$};
\node(xb) at (-1.4,-1){$111$};
\node(xc) at (-1,-1){$101$};
\node(xab) at (-1.6,-0.8){$211$};
\node(xac) at (-1.2,-0.8){$201$};
\node(xbb) at (-1.4,-0.6){$301$};
\draw[->](xa) edge(xab);
\draw[->](xb) edge(xac);
\draw[->](xbb) edge(xac);
\draw[->](xab) edge(xb);
\draw[->](xac) edge(xc);
\draw[->](xab) edge(xbb);
\node(xxa) at (-1.6,-2){$112$};
\node(xxb) at (-1.2,-2){$102$};
\node(xxab) at (-1.4,-1.8){$202$};
\draw[->](xxa) edge(xxab);
\draw[->](xxab) edge(xxb);
\node(xxxa) at (-1.4,-3){$103$};
\draw[<-](a) edge (xa);
\draw[<-](b) edge (xb);
\draw[<-](c) edge (xc);
\draw[<-](ab) edge (xab);
\draw[<-](ac) edge (xac);
\draw[<-](bb) edge (xbb);

\draw[<-](xa) edge (xxa);
\draw[<-](xb) edge (xxb);
\draw[<-](xab) edge (xxab);

\draw[<-](xxa) edge (xxxa);

\end{tikzpicture}$$

\section{Higher Frieze Patterns}

In the two-dimensional case, a glide reflection is a reflection about a line, composed with a translation along that line. An important property of frieze patterns is that they are invariant under a glide reflection and hence periodic. Pictorially, this means points on a frieze pattern of width $n$ can be parameterised by 2-subsets of $\{1,2,\cdots,n+3\}$, where each subset $\{i,j\}$ has at most four neighbours: $\{i-1,j\}$, $\{i,j-1\}$, $\{i+1,j\}$ and $\{i,j+1\}$, where addition is modulo $n+3$. The case where $n=3$ is shown below:

$$
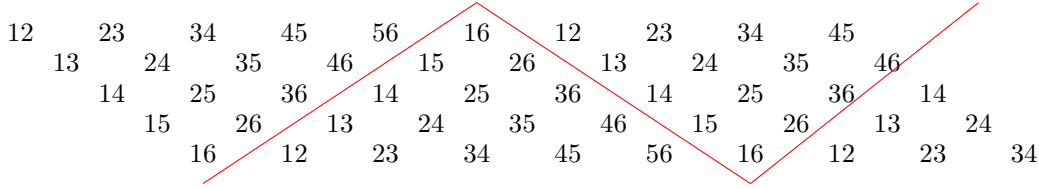
\begin{tikzpicture}[xscale=3]
\node(a) at (0,-0){$12$};
\node(b) at (0.4,-0){$23$};
\node(c) at (0.8,-0){$34$};
\node(d) at (1.2,-0){$45$};
\node(e) at (1.6,-0){$56$};
\node(f) at (2,-0){$16$};
\node(g) at (2.4,-0){$12$};
\node(h) at (2.8,-0){$23$};
\node(i) at (3.2,-0){$34$};
\node(j) at (3.6,-0){$45$};
\node(oa) at (0.2,-0.4){$13$};
\node(ob) at (0.6,-0.4){$24$};
\node(oc) at (1,-0.4){$35$};
\node(od) at (1.4,-0.4){$46$};
\node(oe) at (1.8,-0.4){$15$};
\node(of) at (2.2,-0.4){$26$};
\node(og) at (2.6,-0.4){$13$};
\node(oh) at (3,-0.4){$24$};
\node(oi) at (3.4,-0.4){$35$};
\node(oj) at (3.8,-0.4){$46$};
\node(pa) at (4,-0.8){$14$};
\node(pb) at (0.4,-0.8){$14$};
\node(pc) at (0.8,-0.8){$25$};
\node(pd) at (1.2,-0.8){$36$};
\node(pe) at (1.6,-0.8){$14$};
\node(pf) at (2,-0.8){$25$};
\node(pg) at (2.4,-0.8){$36$};
\node(ph) at (2.8,-0.8){$14$};
\node(pi) at (3.2,-0.8){$25$};
\node(pj) at (3.6,-0.8){$36$};
\node(qa) at (4.2,-1.2){$24$};
\node(qb) at (0.6,-1.2){$15$};
\node(qc) at (1,-1.2){$26$};
\node(qd) at (1.4,-1.2){$13$};
\node(qe) at (1.8,-1.2){$24$};
\node(qf) at (2.2,-1.2){$35$};
\node(qg) at (2.6,-1.2){$46$};
\node(qh) at (3,-1.2){$15$};
\node(qi) at (3.4,-1.2){$26$};
\node(qj) at (3.8,-1.2){$13$};
\node(ra) at (4,-1.6){$23$};
\node(rb) at (4.4,-1.6){$34$};
\node(rc) at (0.8,-1.6){$16$};
\node(rd) at (1.2,-1.6){$12$};
\node(re) at (1.6,-1.6){$23$};
\node(rf) at (2,-1.6){$34$};
\node(rg) at (2.4,-1.6){$45$};
\node(rh) at (2.8,-1.6){$56$};
\node(ri) at (3.2,-1.6){$16$};
\node(rj) at (3.6,-1.6){$12$};
\draw [-,red] (2,0.4) edge (0.8,-2);
\draw [-,red] (2,0.4) edge (3.2,-2);
\draw [-,red] (4.2,0.4) edge (3.2,-2);

\end{tikzpicture}$$
To define a frieze pattern, choose values $\{p_I\}$ for each 2-subset $I$ of $\{1,2,\cdots, n+3\}$ with the conditions that for any for any $i\leq n$, $p_{ii}=0$ and $p_{i(i+1)}=1$ (addition modulo $n+3$).
In this perspective, the unimodular rule is equivalent to the Pl\"ucker relations: any diamond has values given by $$\begin{tikzpicture}[xscale=3]
\node(a) at (0,0){$p_{x(y+1)}$};
\node(b) at (-0.5,-0.5){$p_{xy}$};
\node(c) at (0.5,-0.5){$p_{(x+1)(y+1)}$};
\node(d) at (0,-1){$p_{(x+1)y}$};\end{tikzpicture}$$
So the Pl\"ucker relations $$p_{xy}p_{(x+1)(y+1)}=p_{x(y+1)}p_{(x+1)y}+p_{x(x+1)}p_{y(y+1)}$$ imply the unimodular rule, using $p_{x(x+1)}p_{y(y+1)}=1$ as above. It then makes sense to generalise frieze patterns in the following sense. 
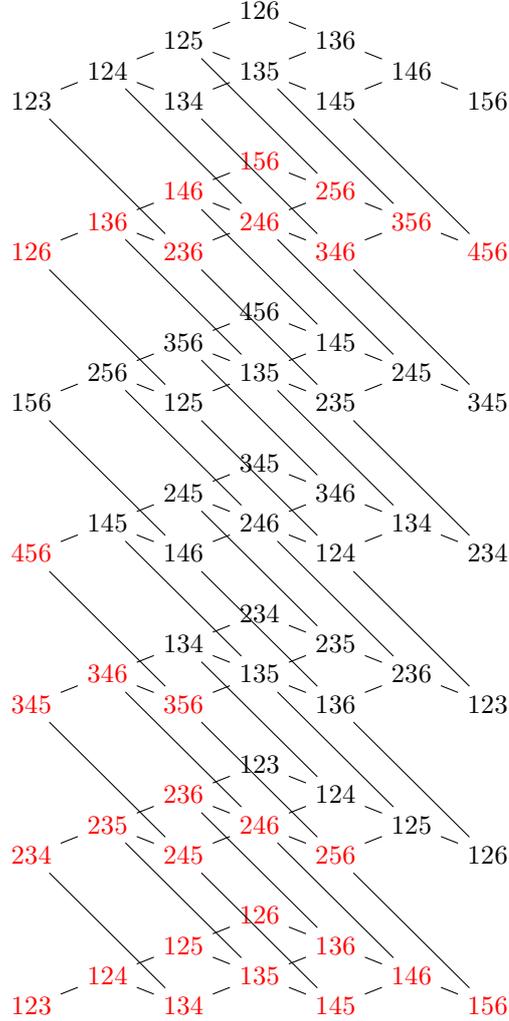
\begin{figure}
\caption{Coordinate labelling of a $(3,6)$-frieze pattern. The red vertices above indicate the cross-sectional triangle at $6$, and the red vertices below show the fundamental domain.}\label{fig1}
\begin{tikzpicture}[xscale=5,yscale=2]

\node(a) at (-2,0)[red]{123};
\node(b) at (-1.6,0)[red]{134};
\node(c) at (-1.2,0)[red]{145};
\node(d) at (-0.8,0)[red]{156};

\node(ab) at (-1.8,0.2)[red]{124};
\node(ac) at (-1.4,0.2)[red]{135};
\node(ad) at (-1,0.2)[red]{146};

\node(bb) at (-1.6,0.4)[red]{125};
\node(bc) at (-1.2,0.4)[red]{136};

\node(cc) at (-1.4,0.6)[red]{126};

\draw[-](a) edge (ab);
\draw[-](b) edge (ac);
\draw[-](c) edge (ad);

\draw[-](bb) edge (ac);
\draw[-](bc) edge (ad);

\draw[-](bb) edge (cc);

\draw[-](ab) edge (b);
\draw[-](ac) edge (c);
\draw[-](ad) edge (d);

\draw[-](ab) edge (bb);
\draw[-](ac) edge (bc);

\draw[-](cc) edge (bc);

\node(xa) at (-2,1)[red]{234};
\node(xb) at (-1.6,1)[red]{245};
\node(xc) at (-1.2,1)[red]{256};
\node(xd) at (-0.8,1){126};

\node(xab) at (-1.8,1.2)[red]{235};
\node(xac) at (-1.4,1.2)[red]{246};
\node(xad) at (-1,1.2){125};

\node(xbb) at (-1.6,1.4)[red]{236};
\node(xbc) at (-1.2,1.4){124};

\node(xcc) at (-1.4,1.6){123};

\draw[-](xa) edge(xab);
\draw[-](xb) edge(xac);
\draw[-](xc) edge(xad);

\draw[-](xbb) edge(xac);
\draw[-](xbc) edge(xad);

\draw[-](xbb) edge(xcc);

\draw[-](xab) edge(xb);
\draw[-](xac) edge(xc);
\draw[-](xad) edge(xd);

\draw[-](xab) edge(xbb);
\draw[-](xac) edge(xbc);

\draw[-](xcc) edge(xbc);

\node(xxa) at (-2,2)[red]{345};
\node(xxb) at (-1.6,2)[red]{356};
\node(xxc) at (-1.2,2){136};
\node(xxd) at (-0.8,2){123};

\node(xxab) at (-1.8,2.2)[red]{346};
\node(xxac) at (-1.4,2.2){135};
\node(xxad) at (-1,2.2){236};

\node(xxbb) at (-1.6,2.4){134};
\node(xxbc) at (-1.2,2.4){235};

\node(xxcc) at (-1.4,2.6){234};

\draw[-](xxa) edge(xxab);
\draw[-](xxb) edge(xxac);
\draw[-](xxc) edge(xxad);

\draw[-](xxbb) edge(xxac);
\draw[-](xxbc) edge(xxad);

\draw[-](xxbb) edge(xxcc);

\draw[-](xxab) edge(xxb);
\draw[-](xxac) edge(xxc);
\draw[-](xxad) edge(xxd);

\draw[-](xxab) edge(xxbb);
\draw[-](xxac) edge(xxbc);

\draw[-](xxcc) edge(xxbc);

\node(xxxa) at (-2,3)[red]{456};
\node(xxxb) at (-1.6,3){146};
\node(xxxc) at (-1.2,3){124};
\node(xxxd) at (-0.8,3){234};

\node(xxxab) at (-1.8,3.2){145};
\node(xxxac) at (-1.4,3.2){246};
\node(xxxad) at (-1,3.2){134};

\node(xxxbb) at (-1.6,3.4){245};
\node(xxxbc) at (-1.2,3.4){346};

\node(xxxcc) at (-1.4,3.6){345};

\draw[-](xxxa) edge(xxxab);
\draw[-](xxxb) edge(xxxac);
\draw[-](xxxc) edge(xxxad);

\draw[-](xxxbb) edge(xxxac);
\draw[-](xxxbc) edge(xxxad);

\draw[-](xxxbb) edge(xxxcc);

\draw[-](xxxab) edge(xxxb);
\draw[-](xxxac) edge(xxxc);
\draw[-](xxxad) edge(xxxd);

\draw[-](xxxab) edge(xxxbb);
\draw[-](xxxac) edge(xxxbc);

\draw[-](xxxcc) edge(xxxbc);

\node(xxxxa) at (-2,4){156};
\node(xxxxb) at (-1.6,4){125};
\node(xxxxc) at (-1.2,4){235};
\node(xxxxd) at (-0.8,4){345};

\node(xxxxab) at (-1.8,4.2){256};
\node(xxxxac) at (-1.4,4.2){135};
\node(xxxxad) at (-1,4.2){245};

\node(xxxxbb) at (-1.6,4.4){356};
\node(xxxxbc) at (-1.2,4.4){145};

\node(xxxxcc) at (-1.4,4.6){456};

\draw[-](xxxxa) edge(xxxxab);
\draw[-](xxxxb) edge(xxxxac);
\draw[-](xxxxc) edge(xxxxad);

\draw[-](xxxxbb) edge(xxxxac);
\draw[-](xxxxbc) edge(xxxxad);

\draw[-](xxxxbb) edge(xxxxcc);

\draw[-](xxxxab) edge(xxxxb);
\draw[-](xxxxac) edge(xxxxc);
\draw[-](xxxxad) edge(xxxxd);

\draw[-](xxxxab) edge(xxxxbb);
\draw[-](xxxxac) edge(xxxxbc);

\draw[-](xxxxcc) edge(xxxxbc);

\node(va) at (-2,5)[red]{126};
\node(vb) at (-1.6,5)[red]{236};
\node(vc) at (-1.2,5)[red]{346};
\node(vd) at (-0.8,5)[red]{456};

\node(vab) at (-1.8,5.2)[red]{136};
\node(vac) at (-1.4,5.2)[red]{246};
\node(vad) at (-1,5.2)[red]{356};

\node(vbb) at (-1.6,5.4)[red]{146};
\node(vbc) at (-1.2,5.4)[red]{256};

\node(vcc) at (-1.4,5.6)[red]{156};

\draw[-](va) edge(vab);
\draw[-](vb) edge(vac);
\draw[-](vc) edge(vad);

\draw[-](vbb) edge(vac);
\draw[-](vbc) edge(vad);

\draw[-](vbb) edge(vcc);

\draw[-](vab) edge(vb);
\draw[-](vac) edge(vc);
\draw[-](vad) edge(vd);

\draw[-](vab) edge(vbb);
\draw[-](vac) edge(vbc);

\draw[-](vcc) edge(vbc);

\node(vxa) at (-2,6){123};
\node(vxb) at (-1.6,6){134};
\node(vxc) at (-1.2,6){145};
\node(vxd) at (-0.8,6){156};

\node(vxab) at (-1.8,6.2){124};
\node(vxac) at (-1.4,6.2){135};
\node(vxad) at (-1,6.2){146};

\node(vxbb) at (-1.6,6.4){125};
\node(vxbc) at (-1.2,6.4){136};

\node(vxcc) at (-1.4,6.6){126};

\draw[-](vxa) edge(vxab);
\draw[-](vxb) edge(vxac);
\draw[-](vxc) edge(vxad);

\draw[-](vxbb) edge(vxac);
\draw[-](vxbc) edge(vxad);

\draw[-](vxbb) edge(vxcc);

\draw[-](vxab) edge(vxb);
\draw[-](vxac) edge(vxc);
\draw[-](vxad) edge(vxd);

\draw[-](vxab) edge(vxbb);
\draw[-](vxac) edge(vxbc);

\draw[-](vxcc) edge(vxbc);

\draw[-](b) edge (xa);
\draw[-](c) edge (xb);
\draw[-](d) edge (xc);
\draw[-](ac) edge (xab);
\draw[-](ad) edge (xac);
\draw[-](bc) edge (xbb);

\draw[-](xb) edge (xxa);
\draw[-](xc) edge (xxb);
\draw[-](xd) edge (xxc);
\draw[-](xac) edge (xxab);
\draw[-](xad) edge (xxac);
\draw[-](xbc) edge (xxbb);

\draw[-](xxb) edge (xxxa);
\draw[-](xxc) edge (xxxb);
\draw[-](xxd) edge (xxxc);
\draw[-](xxac) edge (xxxab);
\draw[-](xxad) edge (xxxac);
\draw[-](xxbc) edge (xxxbb);

\draw[-](xxxb) edge (xxxxa);
\draw[-](xxxc) edge (xxxxb);
\draw[-](xxxd) edge (xxxxc);
\draw[-](xxxac) edge (xxxxab);
\draw[-](xxxad) edge (xxxxac);
\draw[-](xxxbc) edge (xxxxbb);

\draw[-](xxxxb) edge (va);
\draw[-](xxxxc) edge (vb);
\draw[-](xxxxd) edge (vc);
\draw[-](xxxxac) edge (vab);
\draw[-](xxxxad) edge (vac);
\draw[-](xxxxbc) edge (vbb);

\draw[-](vb) edge (vxa);
\draw[-](vc) edge (vxb);
\draw[-](vd) edge (vxc);
\draw[-](vac) edge (vxab);
\draw[-](vad) edge (vxac);
\draw[-](vbc) edge (vxbb);
\end{tikzpicture}
\end{figure}
\begin{defn}\label{knf}
Let $k$ and $n$ be positive integers such that $2\leq k\leq n/2$. A $(k,n)$-frieze pattern, $\mathcal{P}$, is a map from the $k$-multisets of $\{1,2,\cdots,n\}$ to the non-negative integers (sending $I$ to $p_I$) such that:
\begin{itemize}
\item  Each $k$-multiset $I$ with elements from $\{1,2,\cdots,n\}$, is associated a non-negative integer value $p_I$.
\item Each \emph{interval subset}, that is a subset $I=\{i,i+1,\cdots, i+k-1\}$ for a given $1\leq i\leq n$ and under addition modulo $n$, satisfies $p_I=1$.  
\item The set of values $\{p_I\}$ satisfy the $(k,n)$-Pl\"ucker relations.
\end{itemize}
\end{defn}

A $(2,n)$-frieze pattern is just a Coxeter frieze pattern of width $n-3$. The first and final rows consist of only zeroes - these (omitted) rows consist of the coordinates corresponding to a 2-multiset $\{i,i\}$, for some $1\leq i\leq n$, and Equation \ref{eq1} implies that each $p_{ii}=0$. 

\begin{defn}
For $k\geq 3$, we call a $(k,n)$-frieze pattern a \emph{higher frieze pattern}. Let $\mathbf{S}$ be the set of $k$-subsets of $\{1,2,\cdots,n\}$. The \emph{underlying graph} of a $(k,n)$-frieze pattern for $k\geq 3$ is the graph with vertices indexed by the elements of $\mathbf{S}\times \mathbb{Z}$ and consisting of edges between $(I,m)$ and $(J,m)$ wherever $I\setminus\{i\}=J\setminus\{i+1\}$ for some $1\leq i< n$ as well as edges between $(I,m)$ and $(J,m+1)$ wherever $I\setminus\{n\}=J\setminus\{1\}$. The \emph{coordinate label} of a vertex $(I,m)\in\mathbf{S}\times \mathbb{Z}$, where $I=\{i_1,i_2,\cdots,i_k\}$, in the underlying graph of a higher frieze pattern is $i_1i_2\cdots i_k$. A higher frieze pattern is displayed on its underlying graph by setting each vertex $(I,m)\in \mathbf{S}\times \mathbb{Z}$ to have value $p_I$. 

The \emph{fundamental domain} of a (higher) frieze pattern is given by the collection of values $\{p_I\}$ indexed by the $k$-subsets $I\subset\{1,2,\cdots,n\}$. The \emph{underlying graph of the fundamental domain} of a (higher) frieze pattern is the graph with vertices indexed by the elements of $\mathbf{S}$ and consisting of edges between $I$ and $J$ wherever $I\setminus\{i\}=J\setminus\{i+1\}$ for some $1\leq i< n$. \end{defn}

The coordinates of $(3,6)$-frieze pattern and its fundamental domain are given in Figure \ref{fig1}. 
One complicating factor in the study of higher frieze patterns is that they should be visualised in a higher dimension. In contrast, any Pl\"ucker exchange relation should still be visualised on a two dimensional plane. We introduce a convention for $(3,n)$-frieze patterns to make sense of this disparity. 

\begin{defn}
For a given $(3,n)$-frieze pattern and element $x\in \{1,2,\cdots,n\}$, the restriction of the underlying graph of the frieze pattern to the coordinates $\{(I,0)\in\mathbf{S}\times \mathbb{Z}|x\in I\}$ is the \emph{cross-sectional triangle at $x$}. An example of a cross-sectional triangle is given in Figure \ref{fig1}. There is an order on the elements $\{1,2,\cdots,n\}\setminus \{x\}$ given by $x+1<x+2<\cdots<x-1$; denote this order by $i<_x j$. 
\end{defn}
Now we can generalise diamond relations to higher dimensions.\newpage

\begin{defn}\label{gd} Let $\{x,i,j\}$ be a coordinate in the cross-sectional triangle at $x$ with $i<_x j$, and let $l$ and $m$ be integers such that $i+m<_x j-l$. An \emph{$l\times m$-diamond} (or \emph{generalised diamond}) is formed by the points $A=p_{xij}$, $B=p_{xi(j-l)}$, $C=p_{x(i+m)j}$ and $D=p_{x(i+m)(j-l)}$, as indicated in Figure \ref{fig3}. 

Setting $E=p_{xi(i+m)}$ and $F=p_{x(j-l)j}$, then the \emph{generalised diamond relation} holds: $$BC-AD=EF.$$
\end{defn}
The generalised diamond relation is illustrated in Figure \ref{fig3}.

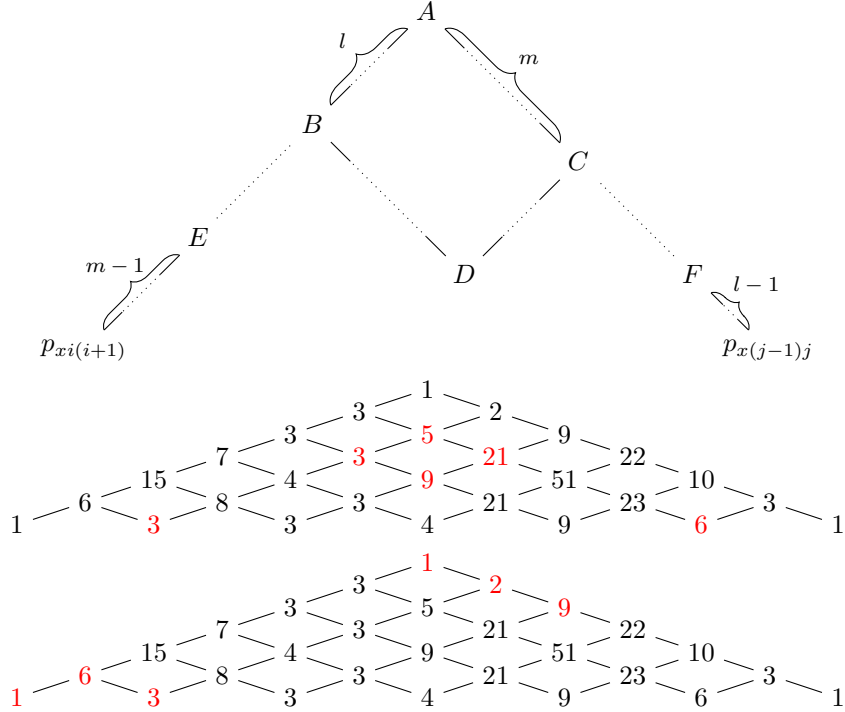
\begin{figure}
\caption{The generalised diamond relation, and two examples coming from a cross-sectional triangle from the $(3,9)$-frieze pattern in Figure \ref{ng2}. The generalised diamond relation depicted above shows $3\times 21-5\times 9=3\times 6$ and the generalised diamond relation depicted below shows $6\times 2-3\times 1=1\times 9$.} \label{fig3} \begin{tikzpicture}
\node(a) at (0,0){$A$};
\node(b) at (-1.5,-1.5){$B$};
\node(c) at (2,-2){$C$};
\node(d) at (0.5,-3.5){$D$};
\node(e) at (-3,-3){$E$};
\node(f) at (3.5,-3.5){$F$};

\node(g) at (-4.5,-4.5){$p_{xi(i+1)}$};
\node(h) at (4.5,-4.5){$p_{x(j-1)j}$};


\draw[-](h) edge (4.1,-4.1);
\draw[-](f) edge (3.9,-3.9);
\draw[-,dotted](4.1,-4.1) edge (3.9,-3.9);


\draw[-](g) edge (-4,-4);
\draw[-,dotted](-3.5,-3.5) edge (-4,-4);
\draw[-](-3.5,-3.5) edge (e);

\draw[-,dotted](e) edge (b);
\draw[-,dotted](f) edge (c);



\draw[-](b) edge (-1,-1);
\draw[-,dotted](-0.5,-0.5) edge (-1,-1);
\draw[-](a) edge (-0.5,-0.5);

\draw[-](d) edge (1,-3);
\draw[-,dotted](1.5,-2.5) edge (1,-3);
\draw[-](c) edge (1.5,-2.5);

\draw[-](a) edge (0.5,-0.5);
\draw[-,dotted](1.5,-1.5) edge (0.5,-0.5);
\draw[-](c) edge (1.5,-1.5);

\draw[-](b) edge (-1,-2);
\draw[-,dotted](0,-3) edge (-1,-2);
\draw[-](d) edge (0,-3);

\draw [decorate,decoration={brace,amplitude=7pt},]
(b) -- (a) node [black,midway,xshift=-0.35cm,yshift=0.35cm] 
{\footnotesize $l$};
\draw [decorate,decoration={brace,amplitude=7pt},]
(a) -- (c) node [black,midway,xshift=0.35cm,yshift=0.35cm] 
{\footnotesize $m$};
\draw [decorate,decoration={brace,amplitude=6pt},]
(g) -- (e) node [black,midway,xshift=-0.35cm,yshift=0.35cm] 
{\footnotesize $m-1$};
\draw [decorate,decoration={brace,amplitude=5pt},]
(f) -- (h) node [black,midway,xshift=0.35cm,yshift=0.35cm] 
{\footnotesize $l-1$};
\end{tikzpicture}
\begin{tikzpicture}[xscale=4.5,yscale=1.5]

\node(a) at (-2,0){1};
\node(b) at (-1.6,0)[red]{3};
\node(c) at (-1.2,0){3};
\node(d) at (-0.8,0){4};
\node(e) at (-0.4,0){9};
\node(f) at (-0,0)[red]{6};
\node(g) at (0.4,0){1};

\node(ab) at (-1.8,0.2){6};
\node(ac) at (-1.4,0.2){8};
\node(ad) at (-1,0.2){3};
\node(ae) at (-0.6,0.2){21};
\node(af) at (-0.2,0.2){23};
\node(ag) at (0.2,0.2){3};

\node(bb) at (-1.6,0.4){15};
\node(bc) at (-1.2,0.4){4};
\node(bd) at (-0.8,0.4)[red]{9};
\node(be) at (-0.4,0.4){51};
\node(bf) at (-0,0.4){10};

\node(cc) at (-1.4,0.6){7};
\node(cd) at (-1,0.6)[red]{3};
\node(ce) at (-0.6,0.6)[red]{21};
\node(cf) at (-0.2,0.6){22};

\node(dc) at (-1.2,0.8){3};
\node(dd) at (-0.8,0.8)[red]{5};
\node(de) at (-0.4,0.8){9};

\node(ed) at (-1,1){3};
\node(ee) at (-0.6,1){2};

\node(fd) at (-0.8,1.2){1};

\draw[-](a) edge (ab);
\draw[-](b) edge (ac);
\draw[-](c) edge (ad);
\draw[-](d) edge (ae);
\draw[-](e) edge (af);
\draw[-](f) edge (ag);

\draw[-](bb) edge (ac);
\draw[-](bc) edge (ad);
\draw[-](bd) edge (ae);
\draw[-](be) edge (af);
\draw[-](bf) edge (ag);

\draw[-](bb) edge (cc);
\draw[-](bc) edge (cd);
\draw[-](bd) edge (ce);
\draw[-](be) edge (cf);

\draw[-](ab) edge (b);
\draw[-](ac) edge (c);
\draw[-](ad) edge (d);
\draw[-](ae) edge (e);
\draw[-](af) edge (f);
\draw[-](ag) edge (g);

\draw[-](ab) edge (bb);
\draw[-](ac) edge (bc);
\draw[-](ad) edge (bd);
\draw[-](ae) edge (be);
\draw[-](af) edge (bf);

\draw[-](cc) edge (bc);
\draw[-](cd) edge (bd);
\draw[-](ce) edge (be);
\draw[-](cf) edge (bf);

\draw[-](cc) edge (dc);
\draw[-](cd) edge (dd);
\draw[-](ce) edge (de);

\draw[-](dc) edge (ed);
\draw[-](dd) edge (ee);

\draw[-](dc) edge (cd);
\draw[-](dd) edge (ce);
\draw[-](de) edge (cf);

\draw[-](ed) edge (dd);
\draw[-](ee) edge (de);

\draw[-](ed) edge (fd);

\draw[-](fd) edge (ee);
\end{tikzpicture}

\begin{tikzpicture}[xscale=4.5,yscale=1.5]

\node(a) at (-2,0)[red]{1};
\node(b) at (-1.6,0)[red]{3};
\node(c) at (-1.2,0){3};
\node(d) at (-0.8,0){4};
\node(e) at (-0.4,0){9};
\node(f) at (-0,0){6};
\node(g) at (0.4,0){1};

\node(ab) at (-1.8,0.2)[red]{6};
\node(ac) at (-1.4,0.2){8};
\node(ad) at (-1,0.2){3};
\node(ae) at (-0.6,0.2){21};
\node(af) at (-0.2,0.2){23};
\node(ag) at (0.2,0.2){3};

\node(bb) at (-1.6,0.4){15};
\node(bc) at (-1.2,0.4){4};
\node(bd) at (-0.8,0.4){9};
\node(be) at (-0.4,0.4){51};
\node(bf) at (-0,0.4){10};

\node(cc) at (-1.4,0.6){7};
\node(cd) at (-1,0.6){3};
\node(ce) at (-0.6,0.6){21};
\node(cf) at (-0.2,0.6){22};

\node(dc) at (-1.2,0.8){3};
\node(dd) at (-0.8,0.8){5};
\node(de) at (-0.4,0.8)[red]{9};

\node(ed) at (-1,1){3};
\node(ee) at (-0.6,1)[red]{2};

\node(fd) at (-0.8,1.2)[red]{1};

\draw[-](a) edge (ab);
\draw[-](b) edge (ac);
\draw[-](c) edge (ad);
\draw[-](d) edge (ae);
\draw[-](e) edge (af);
\draw[-](f) edge (ag);

\draw[-](bb) edge (ac);
\draw[-](bc) edge (ad);
\draw[-](bd) edge (ae);
\draw[-](be) edge (af);
\draw[-](bf) edge (ag);

\draw[-](bb) edge (cc);
\draw[-](bc) edge (cd);
\draw[-](bd) edge (ce);
\draw[-](be) edge (cf);

\draw[-](ab) edge (b);
\draw[-](ac) edge (c);
\draw[-](ad) edge (d);
\draw[-](ae) edge (e);
\draw[-](af) edge (f);
\draw[-](ag) edge (g);

\draw[-](ab) edge (bb);
\draw[-](ac) edge (bc);
\draw[-](ad) edge (bd);
\draw[-](ae) edge (be);
\draw[-](af) edge (bf);

\draw[-](cc) edge (bc);
\draw[-](cd) edge (bd);
\draw[-](ce) edge (be);
\draw[-](cf) edge (bf);

\draw[-](cc) edge (dc);
\draw[-](cd) edge (dd);
\draw[-](ce) edge (de);

\draw[-](dc) edge (ed);
\draw[-](dd) edge (ee);

\draw[-](dc) edge (cd);
\draw[-](dd) edge (ce);
\draw[-](de) edge (cf);

\draw[-](ed) edge (dd);
\draw[-](ee) edge (de);

\draw[-](ed) edge (fd);

\draw[-](fd) edge (ee);
\end{tikzpicture}
\end{figure}

\begin{figure}
\caption{Two examples of geometric $(3,6)$-frieze patterns}\label{geo}
\begin{tikzpicture}[scale=1.5]

\node(a) at (-2,0){1};
\node(b) at (-1.6,0){1};
\node(c) at (-1.2,0){1};
\node(d) at (-0.8,0){1};

\node(ab) at (-1.8,0.2){1};
\node(ac) at (-1.4,0.2){2};
\node(ad) at (-1,0.2){2};

\node(bb) at (-1.6,0.4){1};
\node(bc) at (-1.2,0.4){3};

\node(cc) at (-1.4,0.6){1};

\draw[-](a) edge (ab);
\draw[-](b) edge (ac);
\draw[-](c) edge (ad);

\draw[-](bb) edge (ac);
\draw[-](bc) edge (ad);

\draw[-](bb) edge (cc);

\draw[-](ab) edge (b);
\draw[-](ac) edge (c);
\draw[-](ad) edge (d);

\draw[-](ab) edge (bb);
\draw[-](ac) edge (bc);

\draw[-](cc) edge (bc);

\node(xa) at (-2,1){1};
\node(xb) at (-1.6,1){2};
\node(xc) at (-1.2,1){3};
\node(xd) at (-0.8,1){1};

\node(xab) at (-1.8,1.2){3};
\node(xac) at (-1.4,1.2){5};
\node(xad) at (-1,1.2){1};

\node(xbb) at (-1.6,1.4){6};
\node(xbc) at (-1.2,1.4){1};

\node(xcc) at (-1.4,1.6){1};

\draw[-](xa) edge(xab);
\draw[-](xb) edge(xac);
\draw[-](xc) edge(xad);

\draw[-](xbb) edge(xac);
\draw[-](xbc) edge(xad);

\draw[-](xbb) edge(xcc);

\draw[-](xab) edge(xb);
\draw[-](xac) edge(xc);
\draw[-](xad) edge(xd);

\draw[-](xab) edge(xbb);
\draw[-](xac) edge(xbc);

\draw[-](xcc) edge(xbc);

\node(xxa) at (-2,2){1};
\node(xxb) at (-1.6,2){3};
\node(xxc) at (-1.2,2){3};
\node(xxd) at (-0.8,2){1};

\node(xxab) at (-1.8,2.2){3};
\node(xxac) at (-1.4,2.2){2};
\node(xxad) at (-1,2.2){6};

\node(xxbb) at (-1.6,2.4){1};
\node(xxbc) at (-1.2,2.4){3};

\node(xxcc) at (-1.4,2.6){1};

\draw[-](xxa) edge(xxab);
\draw[-](xxb) edge(xxac);
\draw[-](xxc) edge(xxad);

\draw[-](xxbb) edge(xxac);
\draw[-](xxbc) edge(xxad);

\draw[-](xxbb) edge(xxcc);

\draw[-](xxab) edge(xxb);
\draw[-](xxac) edge(xxc);
\draw[-](xxad) edge(xxd);

\draw[-](xxab) edge(xxbb);
\draw[-](xxac) edge(xxbc);

\draw[-](xxcc) edge(xxbc);

\node(xxxa) at (-2,3){1};
\node(xxxb) at (-1.6,3){2};
\node(xxxc) at (-1.2,3){1};
\node(xxxd) at (-0.8,3){1};

\node(xxxab) at (-1.8,3.2){1};
\node(xxxac) at (-1.4,3.2){5};
\node(xxxad) at (-1,3.2){1};

\node(xxxbb) at (-1.6,3.4){2};
\node(xxxbc) at (-1.2,3.4){3};

\node(xxxcc) at (-1.4,3.6){1};

\draw[-](xxxa) edge(xxxab);
\draw[-](xxxb) edge(xxxac);
\draw[-](xxxc) edge(xxxad);

\draw[-](xxxbb) edge(xxxac);
\draw[-](xxxbc) edge(xxxad);

\draw[-](xxxbb) edge(xxxcc);

\draw[-](xxxab) edge(xxxb);
\draw[-](xxxac) edge(xxxc);
\draw[-](xxxad) edge(xxxd);

\draw[-](xxxab) edge(xxxbb);
\draw[-](xxxac) edge(xxxbc);

\draw[-](xxxcc) edge(xxxbc);

\node(xxxxa) at (-2,4){1};
\node(xxxxb) at (-1.6,4){1};
\node(xxxxc) at (-1.2,4){3};
\node(xxxxd) at (-0.8,4){1};

\node(xxxxab) at (-1.8,4.2){3};
\node(xxxxac) at (-1.4,4.2){2};
\node(xxxxad) at (-1,4.2){2};

\node(xxxxbb) at (-1.6,4.4){3};
\node(xxxxbc) at (-1.2,4.4){1};

\node(xxxxcc) at (-1.4,4.6){1};

\draw[-](xxxxa) edge(xxxxab);
\draw[-](xxxxb) edge(xxxxac);
\draw[-](xxxxc) edge(xxxxad);

\draw[-](xxxxbb) edge(xxxxac);
\draw[-](xxxxbc) edge(xxxxad);

\draw[-](xxxxbb) edge(xxxxcc);

\draw[-](xxxxab) edge(xxxxb);
\draw[-](xxxxac) edge(xxxxc);
\draw[-](xxxxad) edge(xxxxd);

\draw[-](xxxxab) edge(xxxxbb);
\draw[-](xxxxac) edge(xxxxbc);

\draw[-](xxxxcc) edge(xxxxbc);

\node(va) at (-2,5){1};
\node(vb) at (-1.6,5){6};
\node(vc) at (-1.2,5){3};
\node(vd) at (-0.8,5){1};

\node(vab) at (-1.8,5.2){3};
\node(vac) at (-1.4,5.2){5};
\node(vad) at (-1,5.2){3};

\node(vbb) at (-1.6,5.4){2};
\node(vbc) at (-1.2,5.4){3};

\node(vcc) at (-1.4,5.6){1};

\draw[-](va) edge(vab);
\draw[-](vb) edge(vac);
\draw[-](vc) edge(vad);

\draw[-](vbb) edge(vac);
\draw[-](vbc) edge(vad);

\draw[-](vbb) edge(vcc);

\draw[-](vab) edge(vb);
\draw[-](vac) edge(vc);
\draw[-](vad) edge(vd);

\draw[-](vab) edge(vbb);
\draw[-](vac) edge(vbc);

\draw[-](vcc) edge(vbc);

\node(vxa) at (-2,6){1};
\node(vxb) at (-1.6,6){1};
\node(vxc) at (-1.2,6){1};
\node(vxd) at (-0.8,6){1};

\node(vxab) at (-1.8,6.2){1};
\node(vxac) at (-1.4,6.2){2};
\node(vxad) at (-1,6.2){2};

\node(vxbb) at (-1.6,6.4){1};
\node(vxbc) at (-1.2,6.4){3};

\node(vxcc) at (-1.4,6.6){1};

\draw[-](vxa) edge(vxab);
\draw[-](vxb) edge(vxac);
\draw[-](vxc) edge(vxad);

\draw[-](vxbb) edge(vxac);
\draw[-](vxbc) edge(vxad);

\draw[-](vxbb) edge(vxcc);

\draw[-](vxab) edge(vxb);
\draw[-](vxac) edge(vxc);
\draw[-](vxad) edge(vxd);

\draw[-](vxab) edge(vxbb);
\draw[-](vxac) edge(vxbc);

\draw[-](vxcc) edge(vxbc);

\draw[-](b) edge (xa);
\draw[-](c) edge (xb);
\draw[-](d) edge (xc);
\draw[-](ac) edge (xab);
\draw[-](ad) edge (xac);
\draw[-](bc) edge (xbb);

\draw[-](xb) edge (xxa);
\draw[-](xc) edge (xxb);
\draw[-](xd) edge (xxc);
\draw[-](xac) edge (xxab);
\draw[-](xad) edge (xxac);
\draw[-](xbc) edge (xxbb);

\draw[-](xxb) edge (xxxa);
\draw[-](xxc) edge (xxxb);
\draw[-](xxd) edge (xxxc);
\draw[-](xxac) edge (xxxab);
\draw[-](xxad) edge (xxxac);
\draw[-](xxbc) edge (xxxbb);

\draw[-](xxxb) edge (xxxxa);
\draw[-](xxxc) edge (xxxxb);
\draw[-](xxxd) edge (xxxxc);
\draw[-](xxxac) edge (xxxxab);
\draw[-](xxxad) edge (xxxxac);
\draw[-](xxxbc) edge (xxxxbb);

\draw[-](xxxxb) edge (va);
\draw[-](xxxxc) edge (vb);
\draw[-](xxxxd) edge (vc);
\draw[-](xxxxac) edge (vab);
\draw[-](xxxxad) edge (vac);
\draw[-](xxxxbc) edge (vbb);

\draw[-](vb) edge (vxa);
\draw[-](vc) edge (vxb);
\draw[-](vd) edge (vxc);
\draw[-](vac) edge (vxab);
\draw[-](vad) edge (vxac);
\draw[-](vbc) edge (vxbb);

\node(za) at (-4,0){1};
\node(zb) at (-3.6,0){1};
\node(zc) at (-3.2,0){1};
\node(zd) at (-2.8,0){1};

\node(zab) at (-3.8,0.2){1};
\node(zac) at (-3.4,0.2){3};
\node(zad) at (-3,0.2){1};

\node(zbb) at (-3.6,0.4){2};
\node(zbc) at (-3.2,0.4){2};

\node(zcc) at (-3.4,0.6){1};

\draw[-](za) edge (zab);
\draw[-](zb) edge (zac);
\draw[-](zc) edge (zad);

\draw[-](zbb) edge (zac);
\draw[-](zbc) edge (zad);

\draw[-](zbb) edge (zcc);

\draw[-](zab) edge (zb);
\draw[-](zac) edge (zc);
\draw[-](zad) edge (zd);

\draw[-](zab) edge (zbb);
\draw[-](zac) edge (zbc);

\draw[-](zcc) edge (zbc);

\node(zxa) at (-4,1){1};
\node(zxb) at (-3.6,1){2};
\node(zxc) at (-3.2,1){4};
\node(zxd) at (-2.8,1){1};

\node(zxab) at (-3.8,1.2){4};
\node(zxac) at (-3.4,1.2){3};
\node(zxad) at (-3,1.2){2};

\node(zxbb) at (-3.6,1.4){4};
\node(zxbc) at (-3.2,1.4){1};

\node(zxcc) at (-3.4,1.6){1};

\draw[-](zxa) edge(zxab);
\draw[-](zxb) edge(zxac);
\draw[-](zxc) edge(zxad);

\draw[-](zxbb) edge(zxac);
\draw[-](zxbc) edge(zxad);

\draw[-](zxbb) edge(zxcc);

\draw[-](zxab) edge(zxb);
\draw[-](zxac) edge(zxc);
\draw[-](zxad) edge(zxd);

\draw[-](zxab) edge(zxbb);
\draw[-](zxac) edge(zxbc);

\draw[-](zxcc) edge(zxbc);

\node(zxxa) at (-4,2){1};
\node(zxxb) at (-3.6,2){4};
\node(zxxc) at (-3.2,2){2};
\node(zxxd) at (-2.8,2){1};

\node(zxxab) at (-3.8,2.2){2};
\node(zxxac) at (-3.4,2.2){3};
\node(zxxad) at (-3,2.2){4};

\node(zxxbb) at (-3.6,2.4){1};
\node(zxxbc) at (-3.2,2.4){4};

\node(zxxcc) at (-3.4,2.6){1};

\draw[-](zxxa) edge(zxxab);
\draw[-](zxxb) edge(zxxac);
\draw[-](zxxc) edge(zxxad);

\draw[-](zxxbb) edge(zxxac);
\draw[-](zxxbc) edge(zxxad);

\draw[-](zxxbb) edge(zxxcc);

\draw[-](zxxab) edge(zxxb);
\draw[-](zxxac) edge(zxxc);
\draw[-](zxxad) edge(zxxd);

\draw[-](zxxab) edge(zxxbb);
\draw[-](zxxac) edge(zxxbc);

\draw[-](zxxcc) edge(zxxbc);

\node(zxxxa) at (-4,3){1};
\node(zxxxb) at (-3.6,3){1};
\node(zxxxc) at (-3.2,3){1};
\node(zxxxd) at (-2.8,3){1};

\node(zxxxab) at (-3.8,3.2){1};
\node(zxxxac) at (-3.4,3.2){3};
\node(zxxxad) at (-3,3.2){1};

\node(zxxxbb) at (-3.6,3.4){2};
\node(zxxxbc) at (-3.2,3.4){2};

\node(zxxxcc) at (-3.4,3.6){1};

\draw[-](zxxxa) edge(zxxxab);
\draw[-](zxxxb) edge(zxxxac);
\draw[-](zxxxc) edge(zxxxad);

\draw[-](zxxxbb) edge(zxxxac);
\draw[-](zxxxbc) edge(zxxxad);

\draw[-](zxxxbb) edge(zxxxcc);

\draw[-](zxxxab) edge(zxxxb);
\draw[-](zxxxac) edge(zxxxc);
\draw[-](zxxxad) edge(zxxxd);

\draw[-](zxxxab) edge(zxxxbb);
\draw[-](zxxxac) edge(zxxxbc);

\draw[-](zxxxcc) edge(zxxxbc);

\node(zxxxxa) at (-4,4){1};
\node(zxxxxb) at (-3.6,4){2};
\node(zxxxxc) at (-3.2,4){4};
\node(zxxxxd) at (-2.8,4){1};

\node(zxxxxab) at (-3.8,4.2){4};
\node(zxxxxac) at (-3.4,4.2){3};
\node(zxxxxad) at (-3,4.2){2};

\node(zxxxxbb) at (-3.6,4.4){4};
\node(zxxxxbc) at (-3.2,4.4){1};

\node(zxxxxcc) at (-3.4,4.6){1};

\draw[-](zxxxxa) edge(zxxxxab);
\draw[-](zxxxxb) edge(zxxxxac);
\draw[-](zxxxxc) edge(zxxxxad);

\draw[-](zxxxxbb) edge(zxxxxac);
\draw[-](zxxxxbc) edge(zxxxxad);

\draw[-](zxxxxbb) edge(zxxxxcc);

\draw[-](zxxxxab) edge(zxxxxb);
\draw[-](zxxxxac) edge(zxxxxc);
\draw[-](zxxxxad) edge(zxxxxd);

\draw[-](zxxxxab) edge(zxxxxbb);
\draw[-](zxxxxac) edge(zxxxxbc);

\draw[-](zxxxxcc) edge(zxxxxbc);

\node(zva) at (-4,5){1};
\node(zvb) at (-3.6,5){4};
\node(zvc) at (-3.2,5){2};
\node(zvd) at (-2.8,5){1};

\node(zvab) at (-3.8,5.2){2};
\node(zvac) at (-3.4,5.2){3};
\node(zvad) at (-3,5.2){4};

\node(zvbb) at (-3.6,5.4){1};
\node(zvbc) at (-3.2,5.4){4};

\node(zvcc) at (-3.4,5.6){1};

\draw[-](zva) edge(zvab);
\draw[-](zvb) edge(zvac);
\draw[-](zvc) edge(zvad);

\draw[-](zvbb) edge(zvac);
\draw[-](zvbc) edge(zvad);

\draw[-](zvbb) edge(zvcc);

\draw[-](zvab) edge(zvb);
\draw[-](zvac) edge(zvc);
\draw[-](zvad) edge(zvd);

\draw[-](zvab) edge(zvbb);
\draw[-](zvac) edge(zvbc);

\draw[-](zvcc) edge(zvbc);

\node(zvxa) at (-4,6){1};
\node(zvxb) at (-3.6,6){1};
\node(zvxc) at (-3.2,6){1};
\node(zvxd) at (-2.8,6){1};

\node(zvxab) at (-3.8,6.2){1};
\node(zvxac) at (-3.4,6.2){3};
\node(zvxad) at (-3,6.2){1};

\node(zvxbb) at (-3.6,6.4){2};
\node(zvxbc) at (-3.2,6.4){2};

\node(zvxcc) at (-3.4,6.6){1};

\draw[-](zvxa) edge(zvxab);
\draw[-](zvxb) edge(zvxac);
\draw[-](zvxc) edge(zvxad);

\draw[-](zvxbb) edge(zvxac);
\draw[-](zvxbc) edge(zvxad);

\draw[-](zvxbb) edge(zvxcc);

\draw[-](zvxab) edge(zvxb);
\draw[-](zvxac) edge(zvxc);
\draw[-](zvxad) edge(zvxd);

\draw[-](zvxab) edge(zvxbb);
\draw[-](zvxac) edge (zvxbc);

\draw[-](zvxcc) edge (zvxbc);

\draw[-](zb) edge (zxa);
\draw[-](zc) edge (zxb);
\draw[-](zd) edge (zxc);
\draw[-](zac) edge (zxab);
\draw[-](zad) edge (zxac);
\draw[-](zbc) edge (zxbb);

\draw[-](zxb) edge (zxxa);
\draw[-](zxc) edge (zxxb);
\draw[-](zxd) edge (zxxc);
\draw[-](zxac) edge (zxxab);
\draw[-](zxad) edge (zxxac);
\draw[-](zxbc) edge (zxxbb);

\draw[-](zxxb) edge (zxxxa);
\draw[-](zxxc) edge (zxxxb);
\draw[-](zxxd) edge (zxxxc);
\draw[-](zxxac) edge (zxxxab);
\draw[-](zxxad) edge (zxxxac);
\draw[-](zxxbc) edge (zxxxbb);

\draw[-](zxxxb) edge (zxxxxa);
\draw[-](zxxxc) edge (zxxxxb);
\draw[-](zxxxd) edge (zxxxxc);
\draw[-](zxxxac) edge (zxxxxab);
\draw[-](zxxxad) edge (zxxxxac);
\draw[-](zxxxbc) edge (zxxxxbb);

\draw[-](zxxxxb) edge (zva);
\draw[-](zxxxxc) edge (zvb);
\draw[-](zxxxxd) edge (zvc);
\draw[-](zxxxxac) edge (zvab);
\draw[-](zxxxxad) edge (zvac);
\draw[-](zxxxxbc) edge (zvbb);

\draw[-](zvb) edge (zvxa);
\draw[-](zvc) edge (zvxb);
\draw[-](zvd) edge (zvxc);
\draw[-](zvac) edge (zvxab);
\draw[-](zvad) edge (zvxac);
\draw[-](zvbc) edge (zvxbb);

\end{tikzpicture}
\end{figure}
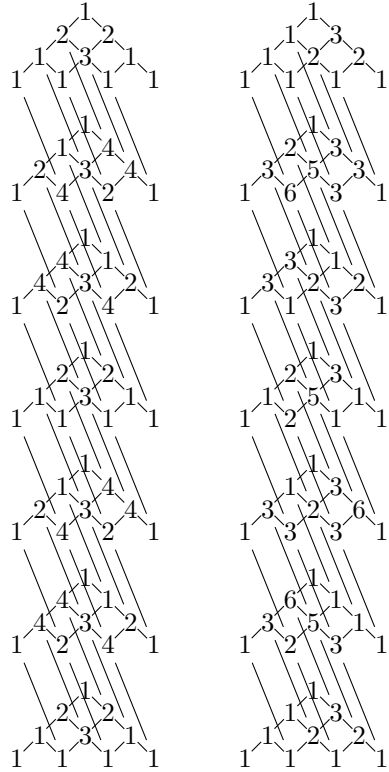
\begin{figure}
\caption{An example of a $(3,9)$-frieze pattern that is not geometric. We have omitted the edges between different cross-sectional triangles.}\label{ng2}
\begin{tikzpicture}[xscale=4.5,yscale=1]
\node(a) at (-2,0){1};
\node(b) at (-1.6,0){3};
\node(c) at (-1.2,0){3};
\node(d) at (-0.8,0){4};
\node(e) at (-0.4,0){9};
\node(f) at (-0,0){6};
\node(g) at (0.4,0){1};

\node(ab) at (-1.8,0.2){6};
\node(ac) at (-1.4,0.2){8};
\node(ad) at (-1,0.2){3};
\node(ae) at (-0.6,0.2){21};
\node(af) at (-0.2,0.2){23};
\node(ag) at (0.2,0.2){3};

\node(bb) at (-1.6,0.4){15};
\node(bc) at (-1.2,0.4){4};
\node(bd) at (-0.8,0.4){9};
\node(be) at (-0.4,0.4){51};
\node(bf) at (-0,0.4){10};

\node(cc) at (-1.4,0.6){7};
\node(cd) at (-1,0.6){3};
\node(ce) at (-0.6,0.6){21};
\node(cf) at (-0.2,0.6){22};

\node(dc) at (-1.2,0.8){3};
\node(dd) at (-0.8,0.8){5};
\node(de) at (-0.4,0.8){9};

\node(ed) at (-1,1){3};
\node(ee) at (-0.6,1){2};

\node(fd) at (-0.8,1.2){1};

\draw[-](a) edge (ab);
\draw[-](b) edge (ac);
\draw[-](c) edge (ad);
\draw[-](d) edge (ae);
\draw[-](e) edge (af);
\draw[-](f) edge (ag);

\draw[-](bb) edge (ac);
\draw[-](bc) edge (ad);
\draw[-](bd) edge (ae);
\draw[-](be) edge (af);
\draw[-](bf) edge (ag);

\draw[-](bb) edge (cc);
\draw[-](bc) edge (cd);
\draw[-](bd) edge (ce);
\draw[-](be) edge (cf);

\draw[-](ab) edge (b);
\draw[-](ac) edge (c);
\draw[-](ad) edge (d);
\draw[-](ae) edge (e);
\draw[-](af) edge (f);
\draw[-](ag) edge (g);

\draw[-](ab) edge (bb);
\draw[-](ac) edge (bc);
\draw[-](ad) edge (bd);
\draw[-](ae) edge (be);
\draw[-](af) edge (bf);

\draw[-](cc) edge (bc);
\draw[-](cd) edge (bd);
\draw[-](ce) edge (be);
\draw[-](cf) edge (bf);

\draw[-](cc) edge (dc);
\draw[-](cd) edge (dd);
\draw[-](ce) edge (de);

\draw[-](dc) edge (ed);
\draw[-](dd) edge (ee);

\draw[-](dc) edge (cd);
\draw[-](dd) edge (ce);
\draw[-](de) edge (cf);

\draw[-](ed) edge (dd);
\draw[-](ee) edge (de);

\draw[-](ed) edge (fd);

\draw[-](fd) edge (ee);

\node(xa) at (-2,2){1};
\node(xb) at (-1.6,2){3};
\node(xc) at (-1.2,2){9};
\node(xd) at (-0.8,2){22};
\node(xe) at (-0.4,2){15};
\node(xf) at (-0,2){3};
\node(xg) at (0.4,2){1};

\node(xab) at (-1.8,2.2){3};
\node(xac) at (-1.4,2.2){5};
\node(xad) at (-1,2.2){51};
\node(xae) at (-0.6,2.2){57};
\node(xaf) at (-0.2,2.2){8};
\node(xag) at (0.2,2.2){3};

\node(xbb) at (-1.6,2.4){2};
\node(xbc) at (-1.2,2.4){21};
\node(xbd) at (-0.8,2.4){126};
\node(xbe) at (-0.4,2.4){26};
\node(xbf) at (-0,2.4){3};

\node(xcc) at (-1.4,2.6){4};
\node(xcd) at (-1,2.6){51};
\node(xce) at (-0.6,2.6){57};
\node(xcf) at (-0.2,2.6){7};

\node(xdc) at (-1.2,2.8){9};
\node(xdd) at (-0.8,2.8){23};
\node(xde) at (-0.4,2.8){15};

\node(xed) at (-1,3){4};
\node(xee) at (-0.6,3){6};

\node(xfd) at (-0.8,3.2){1};

\draw[-](xa) edge (xab);
\draw[-](xb) edge (xac);
\draw[-](xc) edge (xad);
\draw[-](xd) edge (xae);
\draw[-](xe) edge (xaf);
\draw[-](xf) edge (xag);

\draw[-](xbb) edge (xac);
\draw[-](xbc) edge (xad);
\draw[-](xbd) edge (xae);
\draw[-](xbe) edge (xaf);
\draw[-](xbf) edge (xag);

\draw[-](xbb) edge (xcc);
\draw[-](xbc) edge (xcd);
\draw[-](xbd) edge (xce);
\draw[-](xbe) edge (xcf);

\draw[-](xab) edge (xb);
\draw[-](xac) edge (xc);
\draw[-](xad) edge (xd);
\draw[-](xae) edge (xe);
\draw[-](xaf) edge (xf);
\draw[-](xag) edge (xg);

\draw[-](xab) edge (xbb);
\draw[-](xac) edge (xbc);
\draw[-](xad) edge (xbd);
\draw[-](xae) edge (xbe);
\draw[-](xaf) edge (xbf);

\draw[-](xcc) edge (xbc);
\draw[-](xcd) edge (xbd);
\draw[-](xce) edge (xbe);
\draw[-](xcf) edge (xbf);

\draw[-](xcc) edge (xdc);
\draw[-](xcd) edge (xdd);
\draw[-](xce) edge (xde);

\draw[-](xdc) edge (xed);
\draw[-](xdd) edge (xee);

\draw[-](xdc) edge (xcd);
\draw[-](xdd) edge (xce);
\draw[-](xde) edge (xcf);

\draw[-](xed) edge (xdd);
\draw[-](xee) edge (xde);

\draw[-](xed) edge (xfd);

\draw[-](xfd) edge (xee);

\node(xxa) at (-2,4){1};
\node(xxb) at (-1.6,4){4};
\node(xxc) at (-1.2,4){10};
\node(xxd) at (-0.8,4){7};
\node(xxe) at (-0.4,4){2};
\node(xxf) at (-0,4){2};
\node(xxg) at (0.4,4){1};

\node(xxab) at (-1.8,4.2){2};
\node(xxac) at (-1.4,4.2){23};
\node(xxad) at (-1,4.2){26};
\node(xxae) at (-0.6,4.2){4};
\node(xxaf) at (-0.2,4.2){5};
\node(xxag) at (0.2,4.2){4};

\node(xxbb) at (-1.6,4.4){9};
\node(xxbc) at (-1.2,4.4){57};
\node(xxbd) at (-0.8,4.4){12};
\node(xxbe) at (-0.4,4.4){3};
\node(xxbf) at (-0,4.4){9};

\node(xxcc) at (-1.4,4.6){22};
\node(xxcd) at (-1,4.6){26};
\node(xxce) at (-0.6,4.6){4};
\node(xxcf) at (-0.2,4.6){4};

\node(xxdc) at (-1.2,4.8){10};
\node(xxdd) at (-0.8,4.8){8};
\node(xxde) at (-0.4,4.8){2};

\node(xxed) at (-1,5){3};
\node(xxee) at (-0.6,5){3};

\node(xxfd) at (-0.8,5.2){1};

\draw[-](xxa) edge (xxab);
\draw[-](xxb) edge (xxac);
\draw[-](xxc) edge (xxad);
\draw[-](xxd) edge (xxae);
\draw[-](xxe) edge (xxaf);
\draw[-](xxf) edge (xxag);

\draw[-](xxbb) edge (xxac);
\draw[-](xxbc) edge (xxad);
\draw[-](xxbd) edge (xxae);
\draw[-](xxbe) edge (xxaf);
\draw[-](xxbf) edge (xxag);

\draw[-](xxbb) edge (xxcc);
\draw[-](xxbc) edge (xxcd);
\draw[-](xxbd) edge (xxce);
\draw[-](xxbe) edge (xxcf);

\draw[-](xxab) edge (xxb);
\draw[-](xxac) edge (xxc);
\draw[-](xxad) edge (xxd);
\draw[-](xxae) edge (xxe);
\draw[-](xxaf) edge (xxf);
\draw[-](xxag) edge (xxg);

\draw[-](xxab) edge (xxbb);
\draw[-](xxac) edge (xxbc);
\draw[-](xxad) edge (xxbd);
\draw[-](xxae) edge (xxbe);
\draw[-](xxaf) edge (xxbf);

\draw[-](xxcc) edge (xxbc);
\draw[-](xxcd) edge (xxbd);
\draw[-](xxce) edge (xxbe);
\draw[-](xxcf) edge (xxbf);

\draw[-](xxcc) edge (xxdc);
\draw[-](xxcd) edge (xxdd);
\draw[-](xxce) edge (xxde);

\draw[-](xxdc) edge (xxed);
\draw[-](xxdd) edge (xxee);

\draw[-](xxdc) edge (xxcd);
\draw[-](xxdd) edge (xxce);
\draw[-](xxde) edge (xxcf);

\draw[-](xxed) edge (xxdd);
\draw[-](xxee) edge (xxde);

\draw[-](xxed) edge (xxfd);

\draw[-](xxfd) edge (xxee);
\node(va) at (-2,6){1};
\node(vb) at (-1.6,6){3};
\node(vc) at (-1.2,6){3};
\node(vd) at (-0.8,6){4};
\node(ve) at (-0.4,6){9};
\node(vf) at (-0,6){6};
\node(vg) at (0.4,6){1};

\node(vab) at (-1.8,6.2){6};
\node(vac) at (-1.4,6.2){8};
\node(vad) at (-1,6.2){3};
\node(vae) at (-0.6,6.2){21};
\node(vaf) at (-0.2,6.2){23};
\node(vag) at (0.2,6.2){3};

\node(vbb) at (-1.6,6.4){15};
\node(vbc) at (-1.2,6.4){4};
\node(vbd) at (-0.8,6.4){9};
\node(vbe) at (-0.4,6.4){51};
\node(vbf) at (-0,6.4){10};

\node(vcc) at (-1.4,6.6){7};
\node(vcd) at (-1,6.6){3};
\node(vce) at (-0.6,6.6){21};
\node(vcf) at (-0.2,6.6){22};

\node(vdc) at (-1.2,6.8){3};
\node(vdd) at (-0.8,6.8){5};
\node(vde) at (-0.4,6.8){9};

\node(ved) at (-1,7){3};
\node(vee) at (-0.6,7){2};

\node(vfd) at (-0.8,7.2){1};

\draw[-](va) edge (vab);
\draw[-](vb) edge (vac);
\draw[-](vc) edge (vad);
\draw[-](vd) edge (vae);
\draw[-](ve) edge (vaf);
\draw[-](vf) edge (vag);

\draw[-](vbb) edge (vac);
\draw[-](vbc) edge (vad);
\draw[-](vbd) edge (vae);
\draw[-](vbe) edge (vaf);
\draw[-](vbf) edge (vag);

\draw[-](vbb) edge (vcc);
\draw[-](vbc) edge (vcd);
\draw[-](vbd) edge (vce);
\draw[-](vbe) edge (vcf);

\draw[-](vab) edge (vb);
\draw[-](vac) edge (vc);
\draw[-](vad) edge (vd);
\draw[-](vae) edge (ve);
\draw[-](vaf) edge (vf);
\draw[-](vag) edge (vg);

\draw[-](vab) edge (vbb);
\draw[-](vac) edge (vbc);
\draw[-](vad) edge (vbd);
\draw[-](vae) edge (vbe);
\draw[-](vaf) edge (vbf);

\draw[-](vcc) edge (vbc);
\draw[-](vcd) edge (vbd);
\draw[-](vce) edge (vbe);
\draw[-](vcf) edge (vbf);

\draw[-](vcc) edge (vdc);
\draw[-](vcd) edge (vdd);
\draw[-](vce) edge (vde);

\draw[-](vdc) edge (ved);
\draw[-](vdd) edge (vee);

\draw[-](vdc) edge (vcd);
\draw[-](vdd) edge (vce);
\draw[-](vde) edge (vcf);

\draw[-](ved) edge (vdd);
\draw[-](vee) edge (vde);

\draw[-](ved) edge (vfd);

\draw[-](vfd) edge (vee);

\node(vxa) at (-2,8){1};
\node(vxb) at (-1.6,8){3};
\node(vxc) at (-1.2,8){9};
\node(vxd) at (-0.8,8){22};
\node(vxe) at (-0.4,8){15};
\node(vxf) at (-0,8){3};
\node(vxg) at (0.4,8){1};

\node(vxab) at (-1.8,8.2){3};
\node(vxac) at (-1.4,8.2){5};
\node(vxad) at (-1,8.2){51};
\node(vxae) at (-0.6,8.2){57};
\node(vxaf) at (-0.2,8.2){8};
\node(vxag) at (0.2,8.2){3};

\node(vxbb) at (-1.6,8.4){2};
\node(vxbc) at (-1.2,8.4){21};
\node(vxbd) at (-0.8,8.4){126};
\node(vxbe) at (-0.4,8.4){26};
\node(vxbf) at (-0,8.4){3};

\node(vxcc) at (-1.4,8.6){4};
\node(vxcd) at (-1,8.6){51};
\node(vxce) at (-0.6,8.6){57};
\node(vxcf) at (-0.2,8.6){7};

\node(vxdc) at (-1.2,8.8){9};
\node(vxdd) at (-0.8,8.8){23};
\node(vxde) at (-0.4,8.8){15};

\node(vxed) at (-1,9){4};
\node(vxee) at (-0.6,9){6};

\node(vxfd) at (-0.8,9.2){1};

\draw[-](vxa) edge (vxab);
\draw[-](vxb) edge (vxac);
\draw[-](vxc) edge (vxad);
\draw[-](vxd) edge (vxae);
\draw[-](vxe) edge (vxaf);
\draw[-](vxf) edge (vxag);

\draw[-](vxbb) edge (vxac);
\draw[-](vxbc) edge (vxad);
\draw[-](vxbd) edge (vxae);
\draw[-](vxbe) edge (vxaf);
\draw[-](vxbf) edge (vxag);

\draw[-](vxbb) edge (vxcc);
\draw[-](vxbc) edge (vxcd);
\draw[-](vxbd) edge (vxce);
\draw[-](vxbe) edge (vxcf);

\draw[-](vxab) edge (vxb);
\draw[-](vxac) edge (vxc);
\draw[-](vxad) edge (vxd);
\draw[-](vxae) edge (vxe);
\draw[-](vxaf) edge (vxf);
\draw[-](vxag) edge (vxg);

\draw[-](vxab) edge (vxbb);
\draw[-](vxac) edge (vxbc);
\draw[-](vxad) edge (vxbd);
\draw[-](vxae) edge (vxbe);
\draw[-](vxaf) edge (vxbf);

\draw[-](vxcc) edge (vxbc);
\draw[-](vxcd) edge (vxbd);
\draw[-](vxce) edge (vxbe);
\draw[-](vxcf) edge (vxbf);

\draw[-](vxcc) edge (vxdc);
\draw[-](vxcd) edge (vxdd);
\draw[-](vxce) edge (vxde);

\draw[-](vxdc) edge (vxed);
\draw[-](vxdd) edge (vxee);

\draw[-](vxdc) edge (vxcd);
\draw[-](vxdd) edge (vxce);
\draw[-](vxde) edge (vxcf);

\draw[-](vxed) edge (vxdd);
\draw[-](vxee) edge (vxde);

\draw[-](vxed) edge (vxfd);

\draw[-](vxfd) edge (vxee);

\node(vxxa) at (-2,10){1};
\node(vxxb) at (-1.6,10){4};
\node(vxxc) at (-1.2,10){10};
\node(vxxd) at (-0.8,10){7};
\node(vxxe) at (-0.4,10){2};
\node(vxxf) at (-0,10){2};
\node(vxxg) at (0.4,10){1};

\node(vxxab) at (-1.8,10.2){2};
\node(vxxac) at (-1.4,10.2){23};
\node(vxxad) at (-1,10.2){26};
\node(vxxae) at (-0.6,10.2){4};
\node(vxxaf) at (-0.2,10.2){5};
\node(vxxag) at (0.2,10.2){4};

\node(vxxbb) at (-1.6,10.4){9};
\node(vxxbc) at (-1.2,10.4){57};
\node(vxxbd) at (-0.8,10.4){12};
\node(vxxbe) at (-0.4,10.4){3};
\node(vxxbf) at (-0,10.4){9};

\node(vxxcc) at (-1.4,10.6){22};
\node(vxxcd) at (-1,10.6){26};
\node(vxxce) at (-0.6,10.6){4};
\node(vxxcf) at (-0.2,10.6){4};

\node(vxxdc) at (-1.2,10.8){10};
\node(vxxdd) at (-0.8,10.8){8};
\node(vxxde) at (-0.4,10.8){2};

\node(vxxed) at (-1,11){3};
\node(vxxee) at (-0.6,11){3};

\node(vxxfd) at (-0.8,11.2){1};

\draw[-](vxxa) edge (vxxab);
\draw[-](vxxb) edge (vxxac);
\draw[-](vxxc) edge (vxxad);
\draw[-](vxxd) edge (vxxae);
\draw[-](vxxe) edge (vxxaf);
\draw[-](vxxf) edge (vxxag);

\draw[-](vxxbb) edge (vxxac);
\draw[-](vxxbc) edge (vxxad);
\draw[-](vxxbd) edge (vxxae);
\draw[-](vxxbe) edge (vxxaf);
\draw[-](vxxbf) edge (vxxag);

\draw[-](vxxbb) edge (vxxcc);
\draw[-](vxxbc) edge (vxxcd);
\draw[-](vxxbd) edge (vxxce);
\draw[-](vxxbe) edge (vxxcf);

\draw[-](vxxab) edge (vxxb);
\draw[-](vxxac) edge (vxxc);
\draw[-](vxxad) edge (vxxd);
\draw[-](vxxae) edge (vxxe);
\draw[-](vxxaf) edge (vxxf);
\draw[-](vxxag) edge (vxxg);

\draw[-](vxxab) edge (vxxbb);
\draw[-](vxxac) edge (vxxbc);
\draw[-](vxxad) edge (vxxbd);
\draw[-](vxxae) edge (vxxbe);
\draw[-](vxxaf) edge (vxxbf);

\draw[-](vxxcc) edge (vxxbc);
\draw[-](vxxcd) edge (vxxbd);
\draw[-](vxxce) edge (vxxbe);
\draw[-](vxxcf) edge (vxxbf);

\draw[-](vxxcc) edge (vxxdc);
\draw[-](vxxcd) edge (vxxdd);
\draw[-](vxxce) edge (vxxde);

\draw[-](vxxdc) edge (vxxed);
\draw[-](vxxdd) edge (vxxee);

\draw[-](vxxdc) edge (vxxcd);
\draw[-](vxxdd) edge (vxxce);
\draw[-](vxxde) edge (vxxcf);

\draw[-](vxxed) edge (vxxdd);
\draw[-](vxxee) edge (vxxde);

\draw[-](vxxed) edge (vxxfd);

\draw[-](vxxfd) edge (vxxee);

\node(vva) at (-2,12){1};
\node(vvb) at (-1.6,12){3};
\node(vvc) at (-1.2,12){3};
\node(vvd) at (-0.8,12){4};
\node(vve) at (-0.4,12){9};
\node(vvf) at (-0,12){6};
\node(vvg) at (0.4,12){1};

\node(vvab) at (-1.8,12.2){6};
\node(vvac) at (-1.4,12.2){8};
\node(vvad) at (-1,12.2){3};
\node(vvae) at (-0.6,12.2){21};
\node(vvaf) at (-0.2,12.2){23};
\node(vvag) at (0.2,12.2){3};

\node(vvbb) at (-1.6,12.4){15};
\node(vvbc) at (-1.2,12.4){4};
\node(vvbd) at (-0.8,12.4){9};
\node(vvbe) at (-0.4,12.4){51};
\node(vvbf) at (-0,12.4){10};

\node(vvcc) at (-1.4,12.6){7};
\node(vvcd) at (-1,12.6){3};
\node(vvce) at (-0.6,12.6){21};
\node(vvcf) at (-0.2,12.6){22};

\node(vvdc) at (-1.2,12.8){3};
\node(vvdd) at (-0.8,12.8){5};
\node(vvde) at (-0.4,12.8){9};

\node(vved) at (-1,13){3};
\node(vvee) at (-0.6,13){2};

\node(vvfd) at (-0.8,13.2){1};

\draw[-](vva) edge (vvab);
\draw[-](vvb) edge (vvac);
\draw[-](vvc) edge (vvad);
\draw[-](vvd) edge (vvae);
\draw[-](vve) edge (vvaf);
\draw[-](vvf) edge (vvag);

\draw[-](vvbb) edge (vvac);
\draw[-](vvbc) edge (vvad);
\draw[-](vvbd) edge (vvae);
\draw[-](vvbe) edge (vvaf);
\draw[-](vvbf) edge (vvag);

\draw[-](vvbb) edge (vvcc);
\draw[-](vvbc) edge (vvcd);
\draw[-](vvbd) edge (vvce);
\draw[-](vvbe) edge (vvcf);

\draw[-](vvab) edge (vvb);
\draw[-](vvac) edge (vvc);
\draw[-](vvad) edge (vvd);
\draw[-](vvae) edge (vve);
\draw[-](vvaf) edge (vvf);
\draw[-](vvag) edge (vvg);

\draw[-](vvab) edge (vvbb);
\draw[-](vvac) edge (vvbc);
\draw[-](vvad) edge (vvbd);
\draw[-](vvae) edge (vvbe);
\draw[-](vvaf) edge (vvbf);

\draw[-](vvcc) edge (vvbc);
\draw[-](vvcd) edge (vvbd);
\draw[-](vvce) edge (vvbe);
\draw[-](vvcf) edge (vvbf);

\draw[-](vvcc) edge (vvdc);
\draw[-](vvcd) edge (vvdd);
\draw[-](vvce) edge (vvde);

\draw[-](vvdc) edge (vved);
\draw[-](vvdd) edge (vvee);

\draw[-](vvdc) edge (vvcd);
\draw[-](vvdd) edge (vvce);
\draw[-](vvde) edge (vvcf);

\draw[-](vved) edge (vvdd);
\draw[-](vvee) edge (vvde);

\draw[-](vved) edge (vvfd);

\draw[-](vvfd) edge (vvee);

\node(vvxa) at (-2,14){1};
\node(vvxb) at (-1.6,14){3};
\node(vvxc) at (-1.2,14){9};
\node(vvxd) at (-0.8,14){22};
\node(vvxe) at (-0.4,14){15};
\node(vvxf) at (-0,14){3};
\node(vvxg) at (0.4,14){1};

\node(vvxab) at (-1.8,14.2){3};
\node(vvxac) at (-1.4,14.2){5};
\node(vvxad) at (-1,14.2){51};
\node(vvxae) at (-0.6,14.2){57};
\node(vvxaf) at (-0.2,14.2){8};
\node(vvxag) at (0.2,14.2){3};

\node(vvxbb) at (-1.6,14.4){2};
\node(vvxbc) at (-1.2,14.4){21};
\node(vvxbd) at (-0.8,14.4){126};
\node(vvxbe) at (-0.4,14.4){26};
\node(vvxbf) at (-0,14.4){3};

\node(vvxcc) at (-1.4,14.6){4};
\node(vvxcd) at (-1,14.6){51};
\node(vvxce) at (-0.6,14.6){57};
\node(vvxcf) at (-0.2,14.6){7};

\node(vvxdc) at (-1.2,14.8){9};
\node(vvxdd) at (-0.8,14.8){23};
\node(vvxde) at (-0.4,14.8){15};

\node(vvxed) at (-1,15){4};
\node(vvxee) at (-0.6,15){6};

\node(vvxfd) at (-0.8,15.2){1};

\draw[-](vvxa) edge (vvxab);
\draw[-](vvxb) edge (vvxac);
\draw[-](vvxc) edge (vvxad);
\draw[-](vvxd) edge (vvxae);
\draw[-](vvxe) edge (vvxaf);
\draw[-](vvxf) edge (vvxag);

\draw[-](vvxbb) edge (vvxac);
\draw[-](vvxbc) edge (vvxad);
\draw[-](vvxbd) edge (vvxae);
\draw[-](vvxbe) edge (vvxaf);
\draw[-](vvxbf) edge (vvxag);

\draw[-](vvxbb) edge (vvxcc);
\draw[-](vvxbc) edge (vvxcd);
\draw[-](vvxbd) edge (vvxce);
\draw[-](vvxbe) edge (vvxcf);

\draw[-](vvxab) edge (vvxb);
\draw[-](vvxac) edge (vvxc);
\draw[-](vvxad) edge (vvxd);
\draw[-](vvxae) edge (vvxe);
\draw[-](vvxaf) edge (vvxf);
\draw[-](vvxag) edge (vvxg);

\draw[-](vvxab) edge (vvxbb);
\draw[-](vvxac) edge (vvxbc);
\draw[-](vvxad) edge (vvxbd);
\draw[-](vvxae) edge (vvxbe);
\draw[-](vvxaf) edge (vvxbf);

\draw[-](vvxcc) edge (vvxbc);
\draw[-](vvxcd) edge (vvxbd);
\draw[-](vvxce) edge (vvxbe);
\draw[-](vvxcf) edge (vvxbf);

\draw[-](vvxcc) edge (vvxdc);
\draw[-](vvxcd) edge (vvxdd);
\draw[-](vvxce) edge (vvxde);

\draw[-](vvxdc) edge (vvxed);
\draw[-](vvxdd) edge (vvxee);

\draw[-](vvxdc) edge (vvxcd);
\draw[-](vvxdd) edge (vvxce);
\draw[-](vvxde) edge (vvxcf);

\draw[-](vvxed) edge (vvxdd);
\draw[-](vvxee) edge (vvxde);

\draw[-](vvxed) edge (vvxfd);

\draw[-](vvxfd) edge (vvxee);

\node(vvxxa) at (-2,16){1};
\node(vvxxb) at (-1.6,16){4};
\node(vvxxc) at (-1.2,16){10};
\node(vvxxd) at (-0.8,16){7};
\node(vvxxe) at (-0.4,16){2};
\node(vvxxf) at (-0,16){2};
\node(vvxxg) at (0.4,16){1};

\node(vvxxab) at (-1.8,16.2){2};
\node(vvxxac) at (-1.4,16.2){23};
\node(vvxxad) at (-1,16.2){26};
\node(vvxxae) at (-0.6,16.2){4};
\node(vvxxaf) at (-0.2,16.2){5};
\node(vvxxag) at (0.2,16.2){4};

\node(vvxxbb) at (-1.6,16.4){9};
\node(vvxxbc) at (-1.2,16.4){57};
\node(vvxxbd) at (-0.8,16.4){12};
\node(vvxxbe) at (-0.4,16.4){3};
\node(vvxxbf) at (-0,16.4){9};

\node(vvxxcc) at (-1.4,16.6){22};
\node(vvxxcd) at (-1,16.6){26};
\node(vvxxce) at (-0.6,16.6){4};
\node(vvxxcf) at (-0.2,16.6){4};

\node(vvxxdc) at (-1.2,16.8){10};
\node(vvxxdd) at (-0.8,16.8){8};
\node(vvxxde) at (-0.4,16.8){2};

\node(vvxxed) at (-1,17){3};
\node(vvxxee) at (-0.6,17){3};

\node(vvxxfd) at (-0.8,17.2){1};

\draw[-](vvxxa) edge (vvxxab);
\draw[-](vvxxb) edge (vvxxac);
\draw[-](vvxxc) edge (vvxxad);
\draw[-](vvxxd) edge (vvxxae);
\draw[-](vvxxe) edge (vvxxaf);
\draw[-](vvxxf) edge (vvxxag);

\draw[-](vvxxbb) edge (vvxxac);
\draw[-](vvxxbc) edge (vvxxad);
\draw[-](vvxxbd) edge (vvxxae);
\draw[-](vvxxbe) edge (vvxxaf);
\draw[-](vvxxbf) edge (vvxxag);

\draw[-](vvxxbb) edge (vvxxcc);
\draw[-](vvxxbc) edge (vvxxcd);
\draw[-](vvxxbd) edge (vvxxce);
\draw[-](vvxxbe) edge (vvxxcf);

\draw[-](vvxxab) edge (vvxxb);
\draw[-](vvxxac) edge (vvxxc);
\draw[-](vvxxad) edge (vvxxd);
\draw[-](vvxxae) edge (vvxxe);
\draw[-](vvxxaf) edge (vvxxf);
\draw[-](vvxxag) edge (vvxxg);

\draw[-](vvxxab) edge (vvxxbb);
\draw[-](vvxxac) edge (vvxxbc);
\draw[-](vvxxad) edge (vvxxbd);
\draw[-](vvxxae) edge (vvxxbe);
\draw[-](vvxxaf) edge (vvxxbf);

\draw[-](vvxxcc) edge (vvxxbc);
\draw[-](vvxxcd) edge (vvxxbd);
\draw[-](vvxxce) edge (vvxxbe);
\draw[-](vvxxcf) edge (vvxxbf);

\draw[-](vvxxcc) edge (vvxxdc);
\draw[-](vvxxcd) edge (vvxxdd);
\draw[-](vvxxce) edge (vvxxde);

\draw[-](vvxxdc) edge (vvxxed);
\draw[-](vvxxdd) edge (vvxxee);

\draw[-](vvxxdc) edge (vvxxcd);
\draw[-](vvxxdd) edge (vvxxce);
\draw[-](vvxxde) edge (vvxxcf);

\draw[-](vvxxed) edge (vvxxdd);
\draw[-](vvxxee) edge (vvxxde);

\draw[-](vvxxed) edge (vvxxfd);

\draw[-](vvxxfd) edge (vvxxee);

\end{tikzpicture}\end{figure}

\begin{prop}\label{gyem}
The underlying graph of the fundamental domain of a $(k,n)$-frieze pattern has the same underlying graph as $A_{n-k+1}^{k-1}$.
\end{prop}

\begin{proof}
The fundamental domain of $(k,n)$-frieze patterns consists of the $k$-subsets of $\{1,2,\cdots, n\}$. The proof follows from the observation that the set of vertex labels of a $A_{n-k+1}^{k-1}$ is in bijection with the $k$-subsets of $\{1,2,\cdots,n\}$ under the map $$\phi:(l_1,l_2,\cdots l_{k})\mapsto \{l_{k}+1,l_{k}+l_{k-1}+2, \cdots, l_{k}+l_{k-1}+\cdots +l_2+k-1,l_{k}+l_{k-1}+\cdots +l_1+k-1\}.$$ The condition 
$l_1+l_2+\cdots +l_{k}\leq n-k+1$ means $$l_{k}+l_{k-1}+\cdots +l_1+k-1\leq n.$$
So the map sends each vertex of $A_{n-k+1}^{k-1}$ to a $k$-subset of $\{1,2,\cdots, n\}$ (each element of the subset must be distinct as each $l_1\geq 1$). Conversely, every $k$-subset  $I\subset\{1,2,\cdots, m\}$, where $I=\{x_1,x_2,\cdots,x_{k}\}$, defines a unique vertex of $A_{n-k+1}^{k-1}$ by setting $l_{k}=x_1-1$, $l_1=x_{k}-x_{k-1}$ and $l_{k-i}=x_{i+1}-x_{i}-1$ for $2\leq i\leq k-1$. 

We further claim that the edges in the underlying graph of the fundamental domain of a $(k,n)$-frieze pattern coincide with the arrows of $A_{n-k+1}^{k-1}$. Recall that there is an arrow from $\underline{l}$ to $\underline{m}$ in $A_{n-k+1}^{k-1}$ if $\underline{l}=\underline{m}
+v_i$ for $1\leq i\leq k$. 
it is straightforward to check that $\phi(\underline{m})\setminus\{i\}= \phi(\underline{l})\setminus\{i+1\}$ if and only if $\underline{l}=\underline{m}
+v_{k-i}$.
\end{proof}

Proposition \ref{gyem} can be rephrased to say that the underlying graph of the fundamental domain of a $(k,n)$-frieze pattern has the underlying graph of a \emph{higher Auslander algebra} when $k>2$. It is for this reason that we call a $(k,n)$-frieze pattern for $k>2$ a \emph{higher frieze pattern}.

\begin{rem}
By definition, the generalised diamond relations in a $(3,n)$-frieze pattern are in bijection with a set of generating relations for the $(3,n)$-Pl\"ucker relations. It follows that we may combinatorially define a $(3,n)$-frieze pattern as a map from the cylinder of the Auslander algebra of an $A_{n-2}$ quiver to the positive integers that satisfies the generalised diamond relations. 
\end{rem}

\section{Geometric Frieze Patterns}

\subsection{Connection to $\mathrm{SL}_k$-Frieze Patterns}

It has been observed in \cite[Section 3.2]{mgost}, see also \cite[Section 3.4]{mg}, that any $\mathrm{SL}_k$-frieze of width $n-k-1$ determines a point on the Grassmannian $\mathrm{Gr}(k,n)$. This observation reveals the connection between $\mathrm{SL}_k$-frieze patterns and $(k,n)$-frieze patterns.

Firstly, any $(k,n)$-frieze pattern determines an $\mathrm{SL}_k$-frieze pattern of width $n-k-1$. 
Given a $(k,n)$-frieze pattern $\mathcal{P}$, we obtain an $\mathrm{SL}_k$-frieze pattern by taking as the $i^\mathrm{th}$-column in the array to be $$\{p_{i(i+1)\cdots (i+k-2)(i+k-1)},p_{i(i+1)\cdots (i+k-2)(i+k)},\cdots,p_{i(i+1)\cdots (i+k-2)(i-1)}\}$$
This has width $n-k-1$. Any $k\times k$-matrix in the array is of the following form: 
$$M_{ij}:=
  \begin{bmatrix}
    p_{i(i+1)\cdots (i+k-2)j} & p_{i(i+1)\cdots (i+k-2)(j+1)} &\cdots & p_{i(i+1)\cdots (i+k-2)(j+k-1)} \\
    p_{(i+1)(i+2)\cdots (i+k-1)j}& p_{(i+1)(i+2)\cdots (i+k-1)(j+1)} &\cdots &p_{(i+1)(i+2)\cdots (i+k-1)(j+k-1)}  \\
    \cdots&&&\cdots\\
     p_{(i+k-1)(i+k)\cdots (i+2k-3)j}&  p_{(i+k-1)(i+k)\cdots (i+2k-3)(j+1)} &\cdots & p_{(i+k-1)(i+k)\cdots (i+2k-3)(j+k-1)}
  \end{bmatrix}$$
It has been shown by \cite{bfgst} that $\mathrm{det}(M_{ij})=1.$ Conversely, a $(k,n)$-frieze pattern is obtained from its induced $\mathrm{SL}_k$-frieze pattern in the following fashion. Let 
$$\begin{tikzpicture}[xscale=6, yscale=3]
\node(c) at (0.8,0){$1$};
\node(d) at (1.2,0){$1$};
\node(e) at (1.6,0){$1$};
\node(e) at (2,0){$1$};
\node(qc) at (0.2,0.4){$\cdots$};
\node(ob) at (0.6,0.4){$d_{i+n-k-1,j+1}$};
\node(oc) at (1,0.4){$d_{i+n-k,j+2}$};
\node(od) at (1.4,0.4){$\cdots$};
\node(oe) at (1.8,0.4){$d_{i+n-2,j+k}$};
\node(pb) at (0.4,0.8){$\cdots$};
\node(pc) at (0.8,0.8){$\cdots$};
\node(pe) at (1.6,0.8){$\cdots$};
\node(qa) at (0.2,1.2){$d_{i+1,j+1}$};
\node(qb) at (0.6,1.2){$d_{i+2,j+2}$};
\node(qc) at (1,1.2){$\cdots$};
\node(qd) at (1.4,1.2){$d_{i+k,j+k}$};
\node(qc) at (1.8,1.2){$\cdots$};
\node(qj) at (3.8,1.2){$1$};
\node(ra) at (0,1.6){$1$};
\node(rb) at (0.4,1.6){$1$};
\node(rc) at (0.8,1.6){$1$};
\node(rd) at (1.2,1.6){$1$};
\end{tikzpicture}$$
be part of an $\mathrm{SL}_k$-frieze of width $n-k-1$ that was determined by a $(k,n)$-frieze pattern. Then the $k\times n$-matrix

$$\begin{bmatrix}
1&d_{i+1,j+1}&d_{i+2,j+1}&\cdots&d_{i+n-k-1,j+1}&1&0&\cdots& 0&0\\
0&1&d_{i+2,j+2}&\cdots&d_{i+n-k-1,j+2}&d_{i+n-k-1}&1&0&\cdots&0\\
\cdots&&&&&&&&&\cdots\\
0&\cdots&0&1&d_{i+k,j+k}&\cdots&\cdots&\cdots &d_{i+n-2,j+k}&1
\end{bmatrix}$$
is a (non-unique) point on the Grassmannian $\mathrm{Gr}(k,n)$. Label the columns sequentially with the leftmost column labelled by $1$ and the rightmost column labelled $n$. Then the $(k,n)$-frieze pattern is obtained by setting $p_I$ to be the $k\times k$-minor based on columns with indices in $I$. For more information, see the surveys \cite{lam}, \cite{mg}. 

\subsection{Positive Grassmannian}

Two $k$-subsets $I$ and $J$ of $\{1,2,\cdots, n\}$ are said to be \emph{non-crossing} (sometimes referred to as \emph{weakly separated}, see for example \cite{fp}, \cite{ops}) if there do not exist distinct elements $s<t<u<v$ (ordered modulo $n$) where $s,u\in I\setminus J$ and $t,v\in J\setminus I$. A \emph{cluster of Pl\"ucker coordinates in $\mathcal{A}_{k,n}$} is a maximal collection of pairwise non-crossing $k$-subsets of $\{1,2,\cdots,n\}$. It was proven in \cite{dkk}, \cite{ops} that every cluster of Pl\"ucker coordinates in the cluster structure of $\mathrm{Gr}(k,n)$ has $(k-1)(n-k-1)+n$ members.

A $(k,n)$-frieze pattern determines a point on the Grassmannian. In particular, by definition it determines a point on the Grassmannian such that each Pl\"ucker coordinate has a positive value. This means that a $(k,n)$-frieze pattern determines a point on the \emph{positive Grassmannian}, as defined in \cite{post}. For a $(k,n)$-frieze pattern, an \emph{arrangement of smallest minors in $\mathrm{Gr}^+(k,n)$}, as defined in \cite{fp}, is a collection of Pl\"ucker coordinates $\mathcal{J}$ such that $p_I=1$ for all $I\in \mathcal{J}$.  

\begin{theorem}\cite[Theorem 5.6]{fp}\label{fapo}
For $k\leq 3$, a collection of $k$-subests of $\{1,2,\cdots,n\}$ is an arrangement of smallest minors in $\mathrm{Gr}^+(k,n)$ if and only if it is a collection of pairwise non-crossing $k$-subsets of $\{1,2,\cdots,n\}$.
\end{theorem}

This description does not hold in general, see Remark 12.12 of \cite{fp}.
\subsection{Main Result}
We now turn our attention towards cluster algebras. 

\begin{defn}
A $(k,n)$-frieze pattern is \emph{geometric} if the collection of $k$-subsets $I\subset\{1,2,\cdots,n\}$ that satisfy $p_I=1$ forms a cluster in $\mathcal{A}_{k,n}$. 
\end{defn}

Not every frieze pattern is geometric - to the point that many $(k,n)$-frieze patterns do not contain a non-consecutive subset $I$ with $p_I=1$. For an example of a $(3,9)$ frieze pattern where this happens, see Figure \ref{ng2}.

\begin{theorem}\label{main}
For $2\leq k\leq n/2$, a cluster in $\mathcal{A}_{k,n}$ determines a unique, geometric $(k,n)$-frieze pattern. 
If further $k\leq 3$, this restricts to a bijection between geometric $(k,n)$-frieze patterns and clusters in $\mathcal{A}_{k,n}$. 
\end{theorem}

\begin{proof}

When $k=2$, this is well known, see for example \cite[Section 1.5]{mg}.
In the case $k=3$, we will actually prove that there is a bijection between maximal collections of pairwise non-crossing $k$ subsets of $\{1,2,\cdots,n\}$, $\mathcal{C}$, and geometric $(k,n)$-frieze patterns, $\mathcal{P}$. 

Given a $(3,n)$-frieze pattern $\mathcal{P}$, and any two $3$-subsets $I$ and $J$ with $p_I=p_J=1$, then Theorem \ref{fapo} implies that $I$ and $J$ are non-crossing.
So each geometric $(3,n)$-frieze pattern $\mathcal{P}$ determines a maximal collection of pairwise non-crossing $3$-subsets of $\{1,2,\cdots,n\}$, $\mathcal{C}=\{I|p_I=1\}$.

We are left to show that any maximal set of pairwise non-crossing $k$-subsets of $\{1,2,\cdots,n\}$, $\mathcal {C}$ generates a geometric $(k,n)$-frieze pattern, $\mathcal{P}$. Such a maximal set of pairwise non-crossing $k$-subsets of $\{1,2,\cdots,n\}$ determines a cluster in $\mathcal{A}_{k,n}$. In other words, a unique Laurent polynomial $f_I(\underline{x})$ over the indeterminates $\underline{x}=(x_i|i\in \mathcal {C})$ is associated to each $k$-subset $I\subset\{1,2,\cdots,n\}$. Moreover we must have that $p_I=f_I(\underline{x})$ for any choice of values of the $x_i$. Simply set $x_i=1$ for all $i\in \mathcal {C}$ and it is now a consequence of Theorem \ref{fapo} that this determines a geometric $(k,n)$-frieze pattern. 
\end{proof}

\begin{figure}
\caption{A geometric $(3,6)$-frieze pattern expressed via its Laurent polynomials. The frieze pattern on the left of Figure \ref{geo} can be obtained by setting $x_1=x_2=x_3=x_4=1$}\label{cc3}
\begin{tikzpicture}[xscale=10,yscale=1.5]

\node(za) at (-4,0){\scalebox{0.75}{1}};
\node(zb) at (-3.6,0){\scalebox{0.75}{$x_2$}};
\node(zc) at (-3.2,0){\scalebox{0.75}{$x_3$}};
\node(zd) at (-2.8,0){\scalebox{0.75}{1}};

\node(zab) at (-3.8,0.2){\scalebox{0.75}{$x_1$}};
\node(zac) at (-3.4,0.2){\scalebox{0.75}{$\frac{x_2x_3+x_3x_4+x_1x_2}{x_1x_4}$}};
\node(zad) at (-3,0.2){\scalebox{0.75}{$x_4$}};

\node(zbb) at (-3.6,0.4){\scalebox{0.75}{$\frac{x_1+x_3}{x_4}$}};
\node(zbc) at (-3.2,0.4){\scalebox{0.75}{$\frac{x_2+x_4}{x_1}$}};

\node(zcc) at (-3.4,0.6){\scalebox{0.75}{1}};

\draw[-](za) edge (zab);
\draw[-](zb) edge (zac);
\draw[-](zc) edge (zad);

\draw[-](zbb) edge (zac);
\draw[-](zbc) edge (zad);

\draw[-](zbb) edge (zcc);

\draw[-](zab) edge (zb);
\draw[-](zac) edge (zc);
\draw[-](zad) edge (zd);

\draw[-](zab) edge (zbb);
\draw[-](zac) edge (zbc);

\draw[-](zcc) edge (zbc);

\node(zxa) at (-4,1){\scalebox{0.75}{1}};
\node(zxb) at (-3.6,1){\scalebox{0.75}{$\frac{x_1+x_3}{x_2}$}};
\node(zxc) at (-3.2,1){\scalebox{0.75}{$\frac{x_1x_2+x_1x_4+x_2x_3+x_3x_4}{x_2x_3x_4}$}};
\node(zxd) at (-2.8,1){\scalebox{0.75}{1}};

\node(zxab) at (-3.8,1.2){\scalebox{0.75}{$\frac{x_1x_2+x_1x_4+x_2x_3+x_3x_4}{x_1x_2x_4}$}};
\node(zxac) at (-3.4,1.2){\scalebox{0.75}{$\frac{x_3x_4+x_1x_4+x_1x_2}{x_2x_3}$}};
\node(zxad) at (-3,1.2){\scalebox{0.75}{$\frac{x_1+x_3}{x_4}$}};

\node(zxbb) at (-3.6,1.4){\scalebox{0.75}{$\frac{x_1x_2+x_1x_4+x_2x_3+x_3x_4}{x_1x_2x_3}$}};
\node(zxbc) at (-3.2,1.4){\scalebox{0.75}{$x_1$}};

\node(zxcc) at (-3.4,1.6){\scalebox{0.75}{1}};

\draw[-](zxa) edge(zxab);
\draw[-](zxb) edge(zxac);
\draw[-](zxc) edge(zxad);

\draw[-](zxbb) edge(zxac);
\draw[-](zxbc) edge(zxad);

\draw[-](zxbb) edge(zxcc);

\draw[-](zxab) edge(zxb);
\draw[-](zxac) edge(zxc);
\draw[-](zxad) edge(zxd);

\draw[-](zxab) edge(zxbb);
\draw[-](zxac) edge(zxbc);

\draw[-](zxcc) edge(zxbc);

\node(zxxa) at (-4,2){\scalebox{0.75}{1}};
\node(zxxb) at (-3.6,2){\scalebox{0.75}{$\frac{x_1x_2+x_1x_4+x_2x_3+x_3x_4}{x_1x_3x_4}$}};
\node(zxxc) at (-3.2,2){\scalebox{0.75}{$\frac{x_2+x_4}{x_1}$}};
\node(zxxd) at (-2.8,2){\scalebox{0.75}{1}};

\node(zxxab) at (-3.8,2.2){\scalebox{0.75}{$\frac{x_2+x_4}{x_3}$}};
\node(zxxac) at (-3.4,2.2){\scalebox{0.75}{$\frac{x_2x_3+x_3x_4+x_1x_2}{x_1x_4}$}};
\node(zxxad) at (-3,2.2){\scalebox{0.75}{$\frac{x_1x_2+x_1x_4+x_2x_3+x_3x_4}{x_1x_2x_3}$}};

\node(zxxbb) at (-3.6,2.4){\scalebox{0.75}{$x_2$}};
\node(zxxbc) at (-3.2,2.4){\scalebox{0.75}{$\frac{x_1x_2+x_1x_4+x_2x_3+x_3x_4}{x_1x_2x_4}$}};

\node(zxxcc) at (-3.4,2.6){\scalebox{0.75}{1}};

\draw[-](zxxa) edge(zxxab);
\draw[-](zxxb) edge(zxxac);
\draw[-](zxxc) edge(zxxad);

\draw[-](zxxbb) edge(zxxac);
\draw[-](zxxbc) edge(zxxad);

\draw[-](zxxbb) edge(zxxcc);

\draw[-](zxxab) edge(zxxb);
\draw[-](zxxac) edge(zxxc);
\draw[-](zxxad) edge(zxxd);

\draw[-](zxxab) edge(zxxbb);
\draw[-](zxxac) edge(zxxbc);

\draw[-](zxxcc) edge(zxxbc);

\node(zxxxa) at (-4,3){\scalebox{0.75}{1}};
\node(zxxxb) at (-3.6,3){\scalebox{0.75}{$x_4$}};
\node(zxxxc) at (-3.2,3){\scalebox{0.75}{$x_1$}};
\node(zxxxd) at (-2.8,3){\scalebox{0.75}{1}};

\node(zxxxab) at (-3.8,3.2){\scalebox{0.75}{$x_3$}};
\node(zxxxac) at (-3.4,3.2){\scalebox{0.75}{$\frac{x_3x_4+x_1x_4+x_1x_2}{x_2x_3}$}};
\node(zxxxad) at (-3,3.2){\scalebox{0.75}{$x_2$}};

\node(zxxxbb) at (-3.6,3.4){\scalebox{0.75}{$\frac{x_1+x_3}{x_2}$}};
\node(zxxxbc) at (-3.2,3.4){\scalebox{0.75}{$\frac{x_2+x_4}{x_3}$}};

\node(zxxxcc) at (-3.4,3.6){\scalebox{0.75}{1}};

\draw[-](zxxxa) edge(zxxxab);
\draw[-](zxxxb) edge(zxxxac);
\draw[-](zxxxc) edge(zxxxad);

\draw[-](zxxxbb) edge(zxxxac);
\draw[-](zxxxbc) edge(zxxxad);

\draw[-](zxxxbb) edge(zxxxcc);

\draw[-](zxxxab) edge(zxxxb);
\draw[-](zxxxac) edge(zxxxc);
\draw[-](zxxxad) edge(zxxxd);

\draw[-](zxxxab) edge(zxxxbb);
\draw[-](zxxxac) edge(zxxxbc);

\draw[-](zxxxcc) edge(zxxxbc);

\node(zxxxxa) at (-4,4){\scalebox{0.75}{1}};
\node(zxxxxb) at (-3.6,4){\scalebox{0.75}{$\frac{x_1+x_3}{x_4}$}};
\node(zxxxxc) at (-3.2,4){\scalebox{0.75}{$\frac{x_1x_2+x_1x_4+x_2x_3+x_3x_4}{x_1x_2x_4}$}};
\node(zxxxxd) at (-2.8,4){\scalebox{0.75}{1}};

\node(zxxxxab) at (-3.8,4.2){\scalebox{0.75}{$\frac{x_1x_2+x_1x_4+x_2x_3+x_3x_4}{x_2x_3x_4}$}};
\node(zxxxxac) at (-3.4,4.2){\scalebox{0.75}{$\frac{x_2x_3+x_3x_4+x_1x_2}{x_1x_4}$}};
\node(zxxxxad) at (-3,4.2){\scalebox{0.75}{$\frac{x_1+x_3}{x_2}$}};

\node(zxxxxbb) at (-3.6,4.4){\scalebox{0.75}{$\frac{x_1x_2+x_1x_4+x_2x_3+x_3x_4}{x_1x_3x_4}$}};
\node(zxxxxbc) at (-3.2,4.4){\scalebox{0.75}{$x_3$}};

\node(zxxxxcc) at (-3.4,4.6){\scalebox{0.75}{1}};

\draw[-](zxxxxa) edge(zxxxxab);
\draw[-](zxxxxb) edge(zxxxxac);
\draw[-](zxxxxc) edge(zxxxxad);

\draw[-](zxxxxbb) edge(zxxxxac);
\draw[-](zxxxxbc) edge(zxxxxad);

\draw[-](zxxxxbb) edge(zxxxxcc);

\draw[-](zxxxxab) edge(zxxxxb);
\draw[-](zxxxxac) edge(zxxxxc);
\draw[-](zxxxxad) edge(zxxxxd);

\draw[-](zxxxxab) edge(zxxxxbb);
\draw[-](zxxxxac) edge(zxxxxbc);

\draw[-](zxxxxcc) edge(zxxxxbc);

\node(zva) at (-4,5){\scalebox{0.75}{1}};
\node(zvb) at (-3.6,5){\scalebox{0.75}{$\frac{x_1x_2+x_1x_4+x_2x_3+x_3x_4}{x_1x_2x_3}$}};
\node(zvc) at (-3.2,5){\scalebox{0.75}{$\frac{x_2+x_4}{x_3}$}};
\node(zvd) at (-2.8,5){\scalebox{0.75}{1}};

\node(zvab) at (-3.8,5.2){\scalebox{0.75}{$\frac{x_2+x_4}{x_1}$}};
\node(zvac) at (-3.4,5.2){\scalebox{0.75}{$\frac{x_3x_4+x_1x_4+x_1x_2}{x_2x_3}$}};
\node(zvad) at (-3,5.2){\scalebox{0.75}{$\frac{x_1x_2+x_1x_4+x_2x_3+x_3x_4}{x_1x_3x_4}$}};

\node(zvbb) at (-3.6,5.4){\scalebox{0.75}{$x_4$}};
\node(zvbc) at (-3.2,5.4){\scalebox{0.75}{$\frac{x_1x_2+x_1x_4+x_2x_3+x_3x_4}{x_2x_3x_4}$}};

\node(zvcc) at (-3.4,5.6){\scalebox{0.75}{1}};

\draw[-](zva) edge(zvab);
\draw[-](zvb) edge(zvac);
\draw[-](zvc) edge(zvad);

\draw[-](zvbb) edge(zvac);
\draw[-](zvbc) edge(zvad);

\draw[-](zvbb) edge(zvcc);

\draw[-](zvab) edge(zvb);
\draw[-](zvac) edge(zvc);
\draw[-](zvad) edge(zvd);

\draw[-](zvab) edge(zvbb);
\draw[-](zvac) edge(zvbc);

\draw[-](zvcc) edge(zvbc);

\end{tikzpicture}
\end{figure}
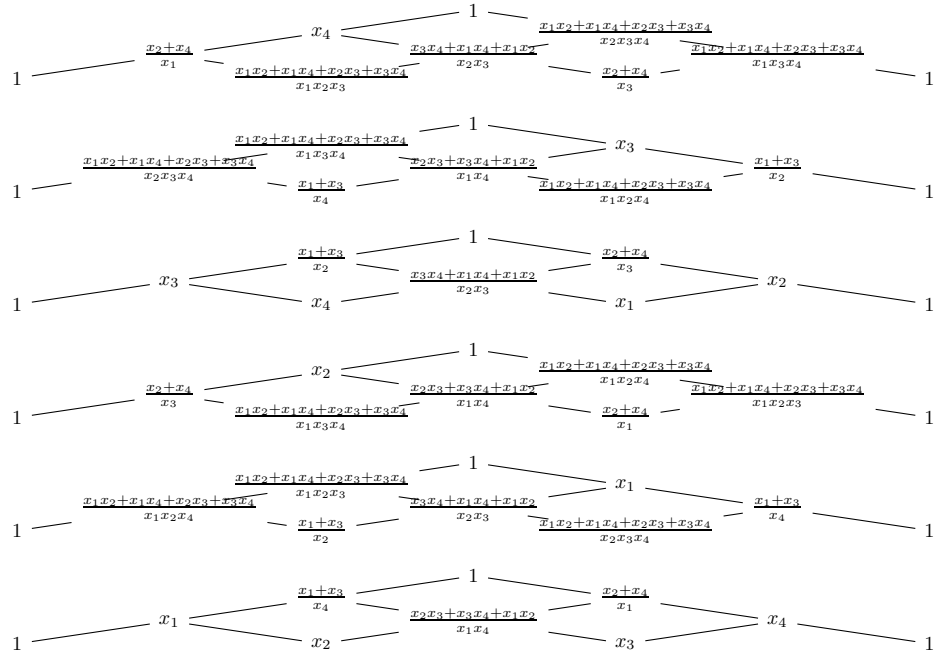

\section{Further Directions}

An early observation is that a (Coxeter) frieze pattern of width $(n-3)$ is completely determined by its first non-trivial row, and that this row consists of a sequence of $n$ integers repeated periodically. A sequence of $n$ integers that induces a frieze pattern (of width $(n-3)$) is called a \emph{quiddity sequence of order $n$}. 
Conway and Coxeter \cite{cc1} used quiddity sequences to proof that the frieze patterns of width $(n-3)$ are in bijection with the triangulations of an $n$-gon.  The proof is elementary, yet insightful. In a sequel paper \cite{mc2}, we determine the class of $\mathrm{SL}_3$-frieze patterns (alternatively $(3,n)$-frieze patterns) for which quiddity sequences have properties analogous to the classical notion. 

\section{Acknowledgements}
This paper was completed with the support of the Austrian Science Fund (FWF): W1230. I would also like to thank my PhD supervisor Karin Baur for her advice and support. In particular, the results presented here are an answer to her question of whether frieze patterns can be generalised in connection with higher Auslander-Reiten theory. 

\bibliographystyle{amsplain}
\bibliography{higherfriezepatterns}

\end{document}